\documentclass[11pt]{amsart}

\usepackage[margin=2.5cm]{geometry}
\usepackage[ruled,vlined]{algorithm2e}
\usepackage{booktabs}
\usepackage{graphicx}
\usepackage[T1]{fontenc}
\usepackage[utf8]{inputenc}
\usepackage{url}
\usepackage{txfonts}
\usepackage{listings}
\usepackage{tikz}
\usepackage{tikz-qtree}

\lstdefinelanguage{GAP}{%
    morekeywords=[2]{and,break,continue,do,elif,else,end,fail,false,fi,for,%
        function,if,in,local,mod,not,od,or,rec,repeat,return,then,true,%
        until,while, %
        AffineSemigroup, Difference, List, Intersection, CatenaryDegreeOfElementInNumericalSemigroup, TameDegreeOfElementInNumericalSemigroup, NumericalSemigroup,IsEmpty, Filtered, ForAny, Gaps, Generators, Rank, Minimum},%
    sensitive=true,%
    morecomment=[l]\#,%
    morestring=[b]',%
    morestring=[b]",%
    basicstyle=\ttfamily,
    }%

\newtheorem{theorem}{Theorem}
\newtheorem{proposition}[theorem]{Proposition}
\newtheorem{lemma}[theorem]{Lemma}

\newtheorem{corollary}[theorem]{Corollary}

\theoremstyle{definition}
\newtheorem{definition}[theorem]{Definition}
\newtheorem{remark}[theorem]{Remark}
\newtheorem{example}[theorem]{Example}

\begin{document}

\title[Algorithms for Generalized Numerical Semigroups]{Algorithms for Generalized Numerical Semigroups}

%\date{}
\author[Cisto]{Carmelo Cisto}
\address{Universit\'{a} di Messina, Dipartimento di Matematica e Informatica, Viale Ferdinando Stagno D'Alcontres 31, 98166 Messina, Italy}
\email{carmelo.cisto@unime.it}

\author[Delgado]{Manuel Delgado}
\address{CMUP, Departamento de Matem\'atica, Faculdade de
	Ci\^encias, Universidade do Porto, Rua do Campo Alegre 687, 4169-007 Porto,
	Portugal} 
\email{mdelgado@fc.up.pt} 

\author[García-Sánchez]{Pedro A. García-Sánchez}
\address{Departamento de Álgebra, Facultad de Ciencias, Campus Fuentenueva s/n, 18071 Granada, España}
\email{pedro@ugr.es}

\thanks{
All the authors acknowledge partial support by the project MTM2017-84890-P, which is funded by Ministerio de Economía y Competitividad and Fondo Europeo de Desarrollo Regional FEDER.
%\\
The second and third authors acknowledge partial support by CMUP (UID/MAT/00144/2013 and UID/MAT/00144/2019), which is funded by FCT (Portugal) with national (MEC) and European structural funds through the programs FEDER, under the partnership agreement PT2020.
%\\
The second author acknowledges a sabbatical grant from the FCT: SFRH/BSAB/142918/2018.
}



\begin{abstract}

We provide algorithms for performing computations in generalized numerical semigroups, that is, submonoids of $\mathbb{N}^{d}$ with finite complement in $\mathbb{N}^{d}$. These semigroups are affine semigroups, which in particular implies that they are finitely generated. For a given finite set of elements in $\mathbb{N}^d$ we show how to deduce if the monoid spanned by this set is a generalized numerical semigroup and, if so, we calculate its set of gaps. Also, given a finite set of elements in $\mathbb{N}^d$ we can determine if it is the set of gaps of a generalized numerical semigroup and, if so, compute the minimal generators of this monoid. 
%We provide algorithms to compute the set of all generalized numerical semigroups with a prescribed genus (the cardinality of their sets of gaps).
We provide a new algorithm to compute the set of all generalized numerical semigroups with a prescribed genus (the cardinality of their sets of gaps). Its implementation allowed us to compute (for various dimensions) the number of numerical semigroups for genus that had not been attained before.
\end{abstract}

%, effective generator, genus, monomial order, algorithm

%\maketitle
\maketitle

\section{Introduction}
Let \(d\) be a positive integer.
A generalized numerical semigroup is a submonoid of $\mathbb{N}^{d}$ with finite complement in $\mathbb{N}^{d}$, where by $\mathbb{N}$ we mean the set of non negative integers. Submonoids of $\mathbb{N}^d$ have been studied by many authors, and in particular those that are finitely generated, which are known in the literature as affine semigroups.

Generalized numerical semigroups were studied in \cite{failla2016algorithms} as a straightforward generalization of the well known concept of numerical semigroup, that is, a submonoid of $\mathbb{N}$ with finite complement in $\mathbb{N}$. If $S$ is a submonoid of $\mathbb{N}^{d}$ with finite complement in $\mathbb{N}^{d}$, then the elements in $\operatorname{H}(S)=\mathbb{N}^{d}\setminus S$ are the \emph{gaps} of $S$, and the \emph{genus} of $S$ is $\operatorname{g}(S)=|\operatorname{H}(S)|$. In \cite{failla2016algorithms}, a procedure to compute the set of generalized numerical semigroups in $\mathbb{N}^d$ with genus $g$ is presented; the cardinality of this set is denoted by $N_{g,d}$. Also it is shown that for fixed genus $g$, the map $d\mapsto N_{g,d}$ is a polynomial of degree~$g$. Other asymptotic properties are presented and several interesting conjectures are raised. In \cite{garcia2018extension}, the number $N_{g,d}$ is computed for several (low) values, and the concept of generalized numerical semigroup is extended to affine semigroups having finite complement in their associated spanning cones. In \cite{Analele} (finite) sets generating generalized numerical semigroups are characterized, and a procedure to compute the set of gaps is presented. 

The aim of this manuscript is to present procedures to determine if a given set of elements in $\mathbb{N}^d$ generates a generalized numerical semigroup and, if so, compute the set of gaps of the resulting monoid. Dually, we present a procedure to determine if a finite set of points in $\mathbb{N}^d$ represents the set of gaps of an affine semigroup, and then we give a method for constructing the minimal generating set of this affine semigroup. We use then these procedures to produce the set of generalized numerical semigroups with given genus. 
Careful implementations of the mentioned procedures have been made in order to make them usable in practice. %As a consequence, some data not previously known has been obtained. 
A section is devoted to these implementations and the new data obtained from them.

The \texttt{GAP} \cite{GAP} package  \texttt{numericalsgps} \cite{numericalsgps} offers tools to deal with numerical and affine semigroups. The possibility of having fast methods to pass from generators to gaps and vice versa allows to describe an affine semigroup by its set of gaps (in the case this set is finite). Membership in a generalized numerical semigroup for which the gaps are known is then trivial, in contrast to membership for an affine semigroup for which only a set of generators is known. Also given an affine semigroup defined by a set of generators one may wonder if it is a generalized numerical semigroup, and thus compute its gaps. Benefiting of the features of the computer algebra system \texttt{GAP}, one can, once computed, add the set of gaps as an attribute of the semigroup (which in particular makes, from that moment on, the membership problem trivial during the \texttt{GAP} session). This opens the possibility of defining new methods for affine semigroups once their gaps are known. In particular, this provides ways to compute special gaps and pseudo-Frobenius elements (the later terminology was used in~\cite{2019arXiv190311028G}). The implementations of these methods is available in \cite{numericalsgps}, in particular in the files \texttt{affine.*} and \texttt{afine-def.*}. The documentation produced for these functions can be consulted in the \texttt{GAP} help system or in the manual of the package.

%For numerical semigroups a large literature has been produced; moreover a package, called \texttt{numericalsgps} \cite{numericalsgps}, has been developed for the computer algebra system \texttt{GAP} \cite{GAP}, in which a lot of algorithms are implemented and it is possible to do many computations on numerical semigroups. The aim of this paper is to provide a brief collection of the firsts basic algorithms concerning generalized numerical semigroups, in order to do computations in this subject, to test properties and conjectures, to produce useful examples. In particular we provide also a new procedure to generate all generalized numerical semigroup in $\mathbb{N}^{d}$ of given genus $g$. 

The paper is structured as follows. In Section~\ref{sec:basics}, we summarize the main definitions and properties concerning generalized numerical semigroups that will be used in the algorithms presented in this manuscript. Sections~\ref{sec:first-alg},~\ref{sec:min-gens} and~\ref{sec:gaps} are devoted to describe three algorithms, one to compute the set of all generalized numerical semigroups of given genus, another to calculate the set of gaps from a minimal generating set of a generalized numerical semigroup, and a procedure to compute the minimal set of generators of a generalized numerical semigroup, once its gaps are known. In Section~\ref{sec:alt-alg}, we provide a new and alternative way to produce all generalized numerical semigroups of given genus, different from that of Section~\ref{sec:first-alg}. Then we have two sections with a rather experimental flavor. First we comment on the implementations of the algorithms; while in the last we gather some computational results obtained.  
%In the last section we gather some computational results obtained by from the implementation of the algorithms.

\section{Basics on Generalized numerical semigroups}\label{sec:basics}

It is known that every submonoid $S$ of $\mathbb{N}^{d}$ admits a unique minimal set of generators $\mathcal{A}(S)$, that is, every element in $S$ is a linear combination of elements in $\mathcal{A}(S)$ with coefficients in $\mathbb{N}$ ($S=\langle \mathcal{A}(S)\rangle$), and no proper subset of $\mathcal{A}(S)$ has this property. It is easy to see that, if $S\subseteq \mathbb{N}^{d}$ is a monoid, the unique minimal set of generators is $\mathcal{A}(S)=S^{*}\setminus(S^{*}+S^{*})$,  where $S^{*}=S\setminus\{0\}$. This set is not always finite, and submonoids $S\subseteq \mathbb{N}^{d}$ for which $\mathcal{A}(S)$ is finite are called \emph{affine semigroups}.

We recall that a \emph{numerical semigroup} is a submonoid $S$ of $\mathbb{N}$ such that $\mathbb{N}\setminus S$ is a finite set. The elements of $\operatorname{H}(S)=\mathbb{N}\setminus S$ are called the \emph{gaps} of $S$ (in some places they are known as \emph{holes}) and the largest integer not in $S$ is the \emph{Frobenius number} of $S$, denoted by $\operatorname{F}(S)$. The number $g=|\operatorname{H}(S)|$ is the \emph{genus} of $S$. 

It is well known that the minimal set of generators of every numerical semigroup $S$ is finite. The elements of $\mathcal{A}(S)$ are sometimes known as the \emph{atoms} (irreducibles, primitive elements) of $S$ or simply minimal generators. Moreover, the fact that $\mathbb{N}\setminus S$ is finite is equivalent to impose that the greatest common divisor of its minimal generating set is one. %The cardinality of $\mathcal{A}(S)$ is the \emph{embedding dimension} of $S$. %Minimal generators greater than the Frobenius number are sometimes called \emph{effective generators}. 
For an introduction to numerical semigroups please refer to \cite{rosales2009numerical}, and also to \cite{ns-app}, where it is explained where the names used above come from. 

 Among affine semigroups, we are interested in \emph{generalized numerical semigroups}, which are submonoids in $\mathbb{N}^{d}$, with $d\geq 2$, and finite complement in $\mathbb{N}^{d}$  (see \cite{failla2016algorithms} and \cite{Analele}). %; this is true in general in any cancellative monoid, \cite{rosales1995finitely}). 
 The notions of \emph{gaps} and $\emph{genus}$ can be then stated for generalized numerical semigroups as for numerical semigroups.

The most important differences between numerical semigroups and generalized numerical semigroups consist in the identification of a Frobenius element and those generators larger than this element: $\mathbb{N}$ has a natural total order while $\mathbb{N}^{d}$ has only a natural partial order (that is induced by the total order in $\mathbb{N}$). The notions of Frobenius element and generators greater than the Frobenius element provide a way to build the semigroup tree of all numerical semigroups up to a given genus (see \cite{bras2009bounds}). The corresponding notions for generalized numerical semigroups are given in \cite{failla2016algorithms} by introducing particular total orders in $\mathbb{N}^{d}$, called \emph{relaxed monomial orders}.

\begin{definition} 
A total order, $\preceq$, on the elements of $\mathbb N^d$ is called a \emph{relaxed monomial order} if it satisfies:
\begin{itemize}
\item[i)] if $\mathbf{v},\mathbf{w}\in \mathbb N^d$ and if $\mathbf{v}\preceq \mathbf{w}$, then $\mathbf{v}\preceq \mathbf{w}+\mathbf{u}$ for all $\mathbf{u}\in \mathbb N^d$;
\item[ii)] if $\mathbf{v}\in \mathbb N^d$ and $\mathbf{v}\neq \mathbf{0}$, then $\mathbf{0}\preceq \mathbf{v}$.
\end{itemize}
\end{definition}

\begin{example} \rm %Examples of relaxed monomial orders
Among the relaxed monomial orders there are the well known \emph{monomial orders} as defined, for instance, in \cite[Section2]{cox2007ideals}. A classical example is the following.  Let $\alpha,\beta \in \mathbb{N}^{d}$, we define $\alpha\preceq \beta$ if and only if  the first nonzero coordinate of $\beta-\alpha$, starting from the left, is positive. The relation $\preceq$ is called \emph{lexicographic order} and it is a relaxed monomial order.

Not all relaxed monomial orders are monomial orders. Take, for instance, 
%In fact, one can prove (REFERENCE here or below?) that the this is the case for the following example.
$\le_{m}$ to be an order on $\mathbb{N}^d$ induced by a monomial order. We define $\mathbf{u}\preceq \mathbf{v}$ if
\begin{itemize}
\item $\min(\mathbf{u})< \min(\mathbf{v})$ or if
\item $\min(\mathbf{u})= \min(\mathbf{v})$ and $\mathbf{u}\le_{m}\mathbf{v}$.
The order $\preceq$ is a relaxed monomial order that is not a monomial order. For instance, in $\mathbb{N}^{3}$ we have $(3,1,1)\preceq (2,7,8)$.  However, $(3,1,1)+(0,2,2)\npreceq (2,7,8)+(0,2,2)$, that is, $\preceq$ is not a monomial order. %(REFERENCE here?).
\end{itemize}
\end{example}

Following the proof of Corollary 6 in \cite[Section 2.4]{cox2007ideals}, we can easily see that a relaxed monomial order in $\mathbb{N}^{d}$ is a well-ordering. So it is possible to state the following general definition.

\begin{definition}
Let $S$ be a submonoid of $\mathbb{N}^{d}$ and $\preceq$ a relaxed monomial order. We define
$\mathbf{m}_{\preceq}(S)$ the smallest element of $S\setminus\{\mathbf{0}\}$ with respect to $\preceq$, and it is called \emph{multiplicity} of $S$ with respect to~$\preceq$.
\end{definition}

For what concerns the notions of Frobenius element and generators greater than the Frobenius element we follow \cite{failla2016algorithms}. %, translated for generalized numerical semigroups, that is what we need for this work. 
It is worth to mention that the same definition of Frobenius element is stated for every submnoids $S$ of $\mathbb{N}^{d}$ in \cite{2019arXiv190311028G}, with the difference that the existence of a Frobenius element is not always guaranteed. In particular, the fact that $S$ is a generalized numerical semigroup is a sufficient condition for the existence of a Frobenius element. A different definition, also inspired in the numerical semigroup setting, is that of \emph{Frobenius vector}, as presented in \cite{assi2015frobenius}. 

\begin{definition} Let $S\subseteq \mathbb{N}^{d}$ be a generalized numerical semigroup. Given a relaxed monomial order $\preceq$ in $\mathbb{N}^{d}$, we define:
\begin{enumerate} 
\item $\mathbf{F}_{\preceq}(S)$ to be the greatest element in $\operatorname{H}(S)$ with respect to $\preceq$, called \emph{Frobenius element} of $S$ with respect to $\preceq$ (we write $\mathbf{F}_\preceq(\mathbb{N}^d)=(-1,\ldots,-1)$ for convenience);
\item $\operatorname{U}_\preceq(S)$ to be the set of minimal generators greater than $\mathbf{F}_{\preceq}(S)$ % called the effective generators of $S$, 
with respect to $\preceq$.
%\item $\mathbf{m}_{\preceq}(S)$ the smallest element of $S\setminus\{\mathbf{0}\}$ with respect to $\preceq$, it is called \emph{multiplicity} of $S$ with respect to~$\preceq$.
\end{enumerate}
\end{definition}

Given a relaxed monomial order in $\mathbb{N}^{d}$, it is possible to arrange the set $\mathcal{S}_{d}$ of all generalized numerical semigroups in $\mathbb{N}^{d}$ as a rooted tree $\mathcal{T}_{\preceq}$, with root in $\mathbb{N}^{d}$ (in the same way as for numerical semigroups, see ~\cite{bras2009bounds}). In particular, we can write an algorithm that provides all generalized numerical semigroups in $\mathbb{N}^{d}$ of a given genus $g$. The following result, easy to prove and stated in a more general setting, is useful in this context. %This is going to be the topic of the rest of this section where we explain this construction in more detail.

\begin{proposition} 
Let $S\subseteq \mathbb{N}^{d}$ be an affine semigroup, and let $\mathbf{v}\in \mathcal{A}(S)$. Then $S\setminus \{\mathbf{v}\}$ is an affine semigroup and it is generated by $\left(\mathcal{A}(S)\setminus \{\mathbf{v}\}\right)\cup \{\mathbf{g}+\mathbf{v}\mid \mathbf{g}\in \mathcal{A}(S)\setminus \{\mathbf{v}\}\}\cup \{2\mathbf{v},3\mathbf{v}\}$. 
\label{gener}
\end{proposition}

\begin{remark} 
The system of generators described in Proposition~\ref{gener} does not have to be a minimal generating system.
\end{remark}

\begin{example}  Let $S=\mathbb{N}^{2}\setminus\{(1,0)\}$. Then $\mathcal{A}(S)=\{(2,0),(3,0),(1,1),(0,1)\}$. Let $\preceq$ be the lexicographic order. So $\mathbf{F}_{\preceq}(S)=(1,0)$, and $(1,1)\in\operatorname{U}_\preceq(S)$. Let $S'=S \setminus  \{(1,1)\}= \mathbb{N}^{2} \setminus \{(1,0),(1,1)\}$. By Proposition~\ref{gener},  $\{(2,0),(3,0),(0,1),(3,1),(4,1),(1,2),(2,2),(3,3)\}$ is a set generating for $S'$. However, $(3,0)+(0,1)=(3,1)$, whence $(3,1)$ is not an atom of $S'$. 
\end{example}%$\mathcal{A}(S')=\{(2,0),(3,0),(1,2),(0,1)\}$.}

Observe that from the definition of relaxed monomial order, for a given submonoid $S\subseteq \mathbb{N}^d$ and a given relaxed monomial order $\preceq$, the multiplicity of $S$ with respect to $\preceq$ is a minimal generator of $S$. 
%This follows easily from the definition of relaxed monomial order. 
Due to its relevance in this work, we state this fact as a lemma.

\begin{lemma}
Let $S$ be a submonoid of $\mathbb{N}^d$, and let $\preceq$ be relaxed monomial order on $\mathbb{N}^d$. Then $\mathbf{m}_\preceq(S)\in \mathcal{A}(S)$.
\end{lemma}
% \begin{proof}
% Notice that $\mathbf{m}_\preceq(S)\preceq \mathbf{s}$ for any $\mathbf{s}\in S\setminus\{\mathbf{0}\}$. Thus $\mathbf{m}_\preceq (S)\precneqq \mathbf{s}+\mathbf{t}$ for any $\mathbf{s},\mathbf{t}\in S\setminus\{0\}$.
% \end{proof}

As a consequence we get the following result.

\begin{corollary}
Let $S$ be a submonoid of $\mathbb{N}^{d}$, and let $\preceq$ be a relaxed monomial order in $\mathbb{N}^{d}$. Then $S\setminus \{\mathbf{m}_{\preceq}(S)\}$ is a submonoid of $\mathbb{N}^d$.
\label{RemoveMult}
\end{corollary}
% \begin{proof}
% Obviously $S\setminus \{\mathbf{m}_{\preceq}(S)\}$ has finite complement in $\mathbb{N}^{d}$. Let $\mathbf{s}_{1},\mathbf{s}_{2}\in S\setminus \{\mathbf{m}_{\preceq}(S)\}$, then $\mathbf{s}_{1}+\mathbf{s}_{2}\in S$ and $\mathbf{s}_{1}+\mathbf{s}_{2}\neq \mathbf{m}_{\preceq}(S)$, since $\mathbf{m}_{\preceq}(S)\preceq \mathbf{s}_{1}$ from whom $\mathbf{m}_{\preceq}(S)\preceq \mathbf{s}_{1}+\mathbf{s}_{2}$.   
% \end{proof}

If $S$ is a numerical semigroup, and we remove one of its minimal generators, then we obtain another numerical semigroup, and the genus increases by one. We can go the other way around by adding gaps to $S$. The only gaps $h$ that make $S\cup \{h\}$ a numerical semigroup are those that are pseudo-Frobenius numbers of $S$ ($h+(S\setminus\{0\})\subseteq S$) fulfilling that $2h\in S$. These gaps are known as \emph{special gaps}, and can be used to construct the set of oversemigroups of a numerical semigroup \cite{rosales2003oversemigroups}. If the generator $a$ we remove from $S$ is greater than its Frobenius number, then this generator becomes trivially the Frobenius number of $S\setminus\{a\}$. Also the Frobenius number of $S$ is a special gap of $S$ (unless $S=\mathbb{N}$). Thus adjoining $\operatorname{F}(S)$ to $S$ produces a numerical semigroup. An analogue for generalized numerical semigroups is given in the following result, that suggests the idea for a simple procedure to generate all generalized numerical semigroups in $\mathbb{N}^{d}$ of given genus $g$.
 
\begin{proposition}\emph{\cite[Proposition 4.1]{failla2016algorithms}} Let $S\subseteq\mathbb{N}^{d}$ be a generalized numerical semigroup, $\preceq$ a relaxed monomial order in $\mathbb{N}^{d}$ and $\mathbf{h}$ a minimal generator. 

\begin{itemize}

\item The set $T=S\cup \{\mathbf{F}_{\preceq}(S)\}$ is a generalized numerical semigroup, moreover $\mathbf{F}_{\preceq}(S)\in \operatorname{U}_\preceq (T)$. 

\item The set $T=S\setminus \{\mathbf{h}\}$ is a generalized numerical semigroup. Furthermore if $\mathbf{h}\in \operatorname{U}_\preceq(S)$, then $\mathbf{F}_{\preceq}(T)=\mathbf{h}$. 
\end{itemize} 
 In particular, every generalized numerical semigroups of genus $g+1$ is obtained from a generalized numerical semigroup $S$ of genus $g$, by removing an element of
$\operatorname{U}_{\preceq}(S)$.

\label{effgenprec} 
\end{proposition}

%The proof of the above result is quite easy. 

Interesting combinatorial properties
%An interesting combinatoric property 
involving the number of generalized numerical semigroups in $\mathbb{N}^{d}$ of a given genus are provided in \cite{failla2016algorithms}. If $A\subseteq \mathbb{N}^{d}$, we denote by $\mathrm{Span}_\mathbb{R}(A)$ the $\mathbb{R}$-vector space spanned by the elements of $A$. %Recall that a \emph{coordinate linear space} is vector subspace generated by a subset of the set of standard basis vectors of $\mathbb{R}^{d}$.

%\begin{proposition}\cite[Proposition 5.2]{failla2016algorithms} Let $S\subseteq \mathbb{N}^{d}$ be a generalized numerical semigroup. Then $\mathrm{Span}_{\mathbb{R}}(\operatorname{H}(S))$ is a coordinate linear space.\label{spaziocoord}
%\end{proposition}

Let $g, d, r\in \mathbb{N}$, with $d$ and $r$ positive. 
We will use the following notation.
\begin{itemize}
\item $\mathcal{S}_{g,d}$ is the set of all generalized numerical semigroups with genus $g$ in $\mathbb{N}^{d}$; $N_{g,d}$ is the cardinality of $\mathcal{S}_{g,d}$.
\item $\mathcal{S}_{g,d}^{(r)}=\left\{S\in \mathcal{S}_{g,d}\mid  \dim(\mathrm{Span}_{\mathbb{R}}(\operatorname{H}(S)))=r\right\}$; $N_{g,d}^{(r)}$ denotes the cardinality of $\mathcal{S}_{g,d}^{(r)}$.
\end{itemize}

\begin{theorem}\cite[Proposition 5.3]{failla2016algorithms} Let $g\in \mathbb{N}$, and define \[F_{g}(x):=\sum_{i=1}^{g}N_{g,i}^{(i)}\binom{x}{i}.\]
Then $F_g(x)$ is a polynomial in  $\mathbb{Q}[x]$ with degree $g$, and $F_{g}(d)=N_{g,d}$ for all $d\in \mathbb{N}$.
\label{polinomi}
\end{theorem}

\section{A first algorithm generating all generalized numerical semigroups of a given genus}\label{sec:first-alg}

Now we can present a first algorithm that computes the number of all generalized numerical semigroup in $\mathbb{N}^{d}$ of a given genus $g$, as explained in~\cite{failla2016algorithms}. The main idea is the following: starting from the trivial generalized numerical semigroup $\mathbb{N}^{d}$ of genus $0$, whose set of generators is the standard basis of the vector space $\mathbb{R}^{d}$, denoted by $\{\mathbf{e}_{1},\ldots,\mathbf{e}_{d}\}$, we produce all generalized numerical semigroups of genus $k$ from the ones of genus $k-1$, for $k$ from $1$ up to $g$. Fixed a relaxed monomial order $\preceq$, for every generalized numerical semigroup $S$ of genus $k-1$, if $\operatorname{U}_\preceq(S)=\{\mathbf{g}_{1},\ldots,\mathbf{g}_{m}\}$, we produce the generalized numerical semigroups $S\setminus \{\mathbf{g}_{1}\},\ldots,S\setminus \{\mathbf{g}_{m}\}$, all of them with genus $k$. In this way all generalized numerical semigroups of genus $k$ are produced. If we consider the directed graph $\mathcal{T}_{\preceq}$ whose set of vertices is the set of all generalized numerical semigroups, and whose edges are the pairs $(S,S\cup \{\mathbf{F}_{\preceq}(S)\})$ (we say that is $S$ is a \emph{son} of $S\cup \{\mathbf{F}_{\preceq}(S)\}$) with $S$ a generalized numerical semigroup, $S\subsetneq \mathbb{N}^{d}$. Remark that, in $\mathcal{T}_\preceq$, the generalized numerical semigroups are produced without redundancy, that is, every generalized numerical semigroup $S\neq \mathbb{N}^{d}$ is the son of a unique generalized numerical semigroup. In fact, if $S$ is the son of both $T_{1}$ and $T_{2}$, then $T_{1}=S\cup \{\mathbf{F}_\preceq(S)\}=T_{2}$.
By using Proposition~\ref{effgenprec}, we have the following result.

\begin{theorem}
Let $\preceq$ be a relaxed monomial order. Then the graph $\mathcal{T}_{\preceq}$ is a rooted tree whose root is $\mathbb{N}^{d}$. Moreover, if $S$ is a generalized numerical semigroup, then the sons of $S$ are $S\setminus \{\mathbf{g}_{1}\},\ldots,S\setminus \{\mathbf{g}_{m}\}$ with $\operatorname{U}_{\preceq}(S)=\{\mathbf{g}_{1},\ldots,\mathbf{g}_{m}\}$, and there is a unique path in $\mathcal{T}_\preceq$ joining $S$ and $\mathbb{N}$.
\end{theorem}

%We remark that, in the tree $\mathcal{T}_\preceq$, the generalized numerical semigroups are produced without redundancy, that is, every generalized numerical semigroups $S\neq \mathbb{N}^{d}$ is the son of a unique generalized numerical semigroup. In fact, if $S$ is the son of both $T_{1}$ and $T_{2}$, then $T_{1}=S\cup \{\mathbf{F}_\preceq(S)\}=T_{2}$.
Recall that we denote by $\mathcal{S}_{g,d}$ the set of generalized numerical semigroups in $\mathbb{N}^{d}$ of genus $g$, and $N_{g,d}=|\mathcal{S}_{g,d}|$. From the above discussion, the computation of $N_{g,d}$ may be done by computing of the sets $\operatorname{U}_\preceq(S)$, for all $S\in \mathcal{S}_{g-1,d}$. So to obtain $N_{g,d}$ we do not need actually to compute every semigroups in $\mathcal{S}_{g,d}$.

\begin{corollary}
Let $g,d\in \mathbb{N}^{d}$ and $\preceq$ be a relaxed monomial order. Then \[N_{g,d}=\sum_{S\in \mathcal{S}_{g-1,d}}|\operatorname{U}_\preceq(S)|.\]
\end{corollary}

\begin{algorithm} 
\DontPrintSemicolon

\KwData{Two integers $g,d\in \mathbb{N}$ and a relaxed monomial order $\preceq$.}
\KwResult{$N_{g,d}$}

\nl $G=\{\mathbf e_1, \mathbf e_2, \dots, \mathbf e_d\}$, $S_{0,d}=\{(\mathbb{N}^{d},G)\}$, $N_{0,d}=1$, $\mathbf{F}_{\preceq}(\mathbb{N}^{d})=(-1,\ldots,-1)$, $\operatorname{H}(\mathbb{N}^{d})=\emptyset$.\; \label{1}
    \For{$i\in\{0,\ldots,g\}$}{
\nl $S_{i,d}=\{(S^{(j)},A^{(j)})\mid j \in \{1,\ldots,N_{i,d}\}\}$, where $\langle A^{(j)}\rangle=S^{(j)}$, $|A^{(j)}|< \infty$.\;\label{2}
  $N_{i+1,d}=0$.\;
     \For{$j\in \{1,\ldots, N_{i,d}\}$}{
    \nl From $(S^{(j)},A^{(j)})$ find out $\mathcal{A}(S^{(j)})$ and $E^{(j)}=\operatorname{U}_\preceq(S^{(j)})$.\;\label{3}
    $N_{i+1,d}=N_{i+1,d}+|E^{(j)}|$.\;
       }
   %\nl  WRITE $N_{i+1,d}$.\;
  %Set $N_{i+1,d}=\sum_{j=1}^{N_{i,d}}|E^{(j)}|$\;
      \nl \If{$i+1=g$}{\Return $N_{g,d}$.\; STOP}  \label{4}
       %$S_{i,d}=\{(S^{(j)},G^{(j)},E^{(j)})\mid j\in\{1,\ldots,N_{i,d}\}\}$.\; 
       $S_{i+1,d}=\emptyset$.\;
       \For{$j\in \{1,\ldots,N_{i,d}\}$}{
            $\{\mathbf{g}_{1},\mathbf{g}_{2},\ldots,\mathbf{g}_{|E^{(j)}|}\}=E^{(j)}$.\;
            \For{$k\in\{1,\ldots,|E^{(j)}|\}$}{
                $S^{(j,k)}=S^{(j)}\setminus \{\mathbf{g}_{k}\}$.\;
                $\mathbf{F}_{\preceq}(S^{(j,k)})=\mathbf{g}_{k}$.\;
                \nl Build $A^{(j,k)}$, with $|A^{(j,k)}|<\infty$, such that $\langle A^{(j,k)} \rangle = S^{(j,k)}$.\; \label{5}
                $\operatorname{H}(S^{(j,k)})=\operatorname{H}(S^{(j)})\cup \{\mathbf{g}_{k}\}$.\;
                  }   
              $S_{i+1,d}=S_{i+1,d}\cup \{(S^{(j,k)},A^{(j,k)})\}$.\;     
            }       
     }

\caption{Algorithm for computing $N_{g,d}$} \label{alg:ngd}
\end{algorithm}

Algorithm~\ref{alg:ngd}  %is the 
consists of
pseudocode for the procedure discussed above for computing $N_{g,d}$. We have labeled some lines of the pseudocode in order to explain them further.

\begin{itemize}

\item Line~\ref{1} contains the initial step of the algorithm, that is, the root of $\mathcal{T}_{\preceq}$. It starts from the trivial generalized numerical semigroup $\mathbb{N}^{d}$, which is finitely generated by the standard-basis vectors and whose set of gaps is empty. Moreover, it is the unique generalized numerical semigroup of genus 0, so $N_{0,d}=1$.% $\mathbb{N}^{d}$ has not a Frobenius element but it is convenient for the algorithm to fix $\mathbf{F}_{\preceq}(\mathbb{N}^{d})=\mathbf{0}$.

\item Line~\ref{2}: $S_{i,d}$ is the set of all generalized numerical semigroups in $\mathbb{N}^{d}$ of genus $i$; we represent each one of these semigroups $S^{(j)}$ by a finite system of generators $A^{(j)}$.

\item Line~\ref{3}: $A^{(j)}$ is not in general the minimal system of generators for $S^{(j)}$; a refinement is needed to find $\mathcal{A}(S^{(j)})$ and $\operatorname{U}_\preceq(S^{(j)})$ to continue the process.

\item Lines~\ref{4} and successive: if $i+1=g$, then $N_{g,d}$ is computed and the algorithm ends. One can compute also the set $\mathcal{S}_{g,d}$, it suffices to run another extra for-loop step. If $i+1<g$, all generalized numerical semigroups of genus $i+1$ are produced from the generalized numerical semigroups of genus $i$, by Proposition~\ref{effgenprec}.

\item  Line~\ref{5}: a finite system of generators $A^{(j,k)}$ for $S^{(j,k)}$ can be obtained from Proposition~\ref{gener}.
\end{itemize}

\begin{remark}
In Line 3 it is required to compute the set of minimal generators of a generalized numerical semigroup $S$, once a finite set of generators $A$ of $S$ is known. This can be performed in different ways. For instance, we can try to decide when an element $\mathbf{v}\in A$ is linear combination of elements in $A\setminus \{\mathbf{v} \}$ with non negative integers coefficients. This passes through determining the existence of nonnegative integer solutions of a linear system of Diophantine equations. Also, we can observe that $\mathbf{v}\in \mathcal{A}(S)$ if and only if $\mathbf{v}-\mathbf{g}\notin S$ for every generator $\mathbf{g}\in A\setminus \{\mathbf{v}\}$. This can be performed recursively, %(as it does \texttt{MinimalGenerators} in the \texttt{GAP} package \texttt{numericalsgps}), 
or in the case we know the gaps of $S$, this test is straightforward. The caveat of this latter approach is memory usage, since we would need to store the set of gaps for all semigroups computed.
\end{remark}

The previous algorithm produces the rooted tree $\mathcal{T}_{\preceq}$, where the root is $\mathbb{N}^{d}$ and the elements of depth $g$ are all generalized numerical semigroups of genus $g$. Different relaxed monomial orders will produce different rooted trees, as the Frobenius elements and the generators greater than them depend on the chosen monomial order. However, with every relaxed monomial order, the sons of all generalized numerical semigroups of genus $g-1$ are going to be all generalized numerical semigroups of genus $g$, generated without redundancy. Thus these trees all have at level $g$ the set of all generalized numerical semigroups with genus $g$.

\section[Algorithm for minimal generators]{An algorithm to compute minimal generators of a generalized numerical semigroup from its set of gaps}\label{sec:min-gens} 

Let $S\subseteq \mathbb{N}^{d}$ be a generalized numerical semigroup and let $\preceq$ be a relaxed monomial order. We suppose that the set $\operatorname{H}(S)$ of gaps is known and we want to find the set \(\mathcal{A}(S)\) of minimal generators of $S$. If $S=\mathbb{N}^{d}$, then we know that $S$ is generated by $\{\mathbf{e}_{1},\mathbf{e}_{2},\dots,\mathbf{e}_{d}\}$, the standard basis vectors of $\mathbb{R}^{d}$. If $S$ is a proper subset of $\mathbb{N}^{d}$, our procedure is based on the following result, that can be found in a more general setting in \cite[Corollary 9]{garcia2013affine}. 

\begin{corollary} 
Let $S\subset \mathbb{N}^{d}$ be a generalized numerical semigroup of genus $g$ and let $\preceq$ be a relaxed monomial order in $\mathbb{N}^{d}$. Suppose that $\operatorname{H}(S)=\{\mathbf{h}_{1}\preceq \mathbf{h}_{2}\preceq \dots \preceq \mathbf{h}_{g}\}$. Then
\begin{itemize}
    \item $\mathbf{h}_{1}$ is a minimal generator of $\mathbb{N}^{d}$, 
    \item for every $i\in\{2,\ldots,g\}$, $S\setminus \{\mathbf{h}_{1},\mathbf{h}_{2},\ldots, \mathbf{h}_{i-1}\}$ is a generalized numerical semigroup, and 
    \item $\mathbf{h}_{i}\in \operatorname{U}_\preceq(S\setminus \{\mathbf{h}_{1},\mathbf{h}_{2},\ldots, \mathbf{h}_{i-1}\})$, for every $i \in \{2, \ldots, g\}$.
\end{itemize}
\end{corollary}
\begin{proof}
It follows easily from Proposition~\ref{effgenprec}, since $\mathbf{h}_{i}=\mathbf{F}_\preceq(S\setminus \{\mathbf{h}_{1},\mathbf{h}_{2},\ldots, \mathbf{h}_{i}\})$.
\end{proof}

\begin{algorithm}[H]
\caption{Algorithm to compute $\mathcal{A}(S)$ from the set $\operatorname{H}(S)$}\label{alg:gfh}
\DontPrintSemicolon
\KwData{A set $H=\{\mathbf{h}_{1},\mathbf{h}_{2},\ldots,\mathbf{h}_{g}\}\subseteq \mathbb{N}^{d}$, a relaxed monomial order $\preceq$}
\KwResult{If $S=\mathbb{N}^{d}\setminus H$ is a generalized numerical semigroup, $\mathcal{A}(S)$ is computed}
\nl $H_{\preceq}=\{\mathbf{h}_{j_{1}}\preceq \mathbf{h}_{j_{2}}\preceq \ldots \preceq \mathbf{h}_{j_{g}}\}$, $lH=\emptyset$, %.\;
 $G=\{\mathbf{e}_{1},\mathbf{e}_{2},\ldots,\mathbf{e}_{d}\}$.\;
\For{$i\in \{1,\ldots,g\}$}{
     \If{$\mathbf{h}_{j_{i}}\notin G$}{
     $\mathbb{N}^{d}\setminus H$ is not a generalized numerical semigroup.\; STOP}
     $lH=lH \cup \{\mathbf{h}_{j_{i}}\}$.\;
     $G=G\setminus \{\mathbf{h}_{j_{i}}\}$\;
  \nl   $A=\{\mathbf{g}+ \mathbf{h}_{j_{i}}\mid \mathbf{g}\in G\}\cup \{2\mathbf{h}_{j_{i}},3\mathbf{h}_{j_{i}}\}$\;
    \For{$\mathbf{v}\in A$}{
   \nl      \If{$\mathbf{v}$ \emph{is a minimal generator of} $\mathbb{N}^{d}\setminus lH$}{
         $G=G\cup \{\mathbf{v}\}$.}
        }
 }
 \Return $G$\;
 
\end{algorithm}

As above, we have labeled some lines in the pseudo code in order to explain how it works.
\begin{itemize}
\item Line 1: we arrange the elements of $H$ with respect to $\preceq$.
\item Line 2: at this point $G\cup A$ is a finite set of generators of $\mathbb{N}^{d}\setminus lH$ (Proposition~\ref{gener}); a refiniement into a minimal generating set is performed in the next steps.
\item Line 3: the elements in $G$ are already minimal generators of $\mathbb{N}^{d}\setminus lH$, so we only have to decice which element in $A$ are minimal generators.
Observe that $\mathbf{v}$ is a minimal generator of a semigroup $S$ if and only if $\mathbf{v}-\mathbf{g}\notin S$ for every generator $\mathbf{g}$ of \(S\), and in this case we have the advantage of knowing the gap set of the semigroup computed.
%Furthrmore consider that: $\mathbf{v}$ is a minimal generator of a semigroup $S$ if and only if $\mathbf{v}-\mathbf{g}\notin S$ for every $\mathbf{g}$ beloging to a set of generators of $S$.
\end{itemize}

Observe that while Algorithm~\ref{alg:ngd} covers all the branches of the tree $\mathcal{T}_{\preceq}$ with respect to $\preceq$ (up to depth $g$), Algorithm~\ref{alg:gfh} walks through the unique branch linking $\mathbb{N}^{d}$ with $S=\mathbb{N}^{d}\setminus H$, in the case $S$ is actually a generalized numerical semigroup such that $\operatorname{H}(S)=H$.
%given by \(H\) as set of gaps.
%from the given set $H$. 
Moreover, Algorithm~\ref{alg:gfh} can verify if a given finite set $H\subseteq \mathbb{N}^{d}$ has the property that $\mathbb{N}^{d}\setminus H$ is a generalized numerical semigroup.\\
%\noindent 
Another way to test if a given set is the set of gaps of a submonoid in $\mathbb{N}^{d}$ is provided by the next result. For $\mathbf{x}\in \mathbb{N}^{d}$ we define
\[%\operatorname{B}(\mathbf{x})=\{\mathbf{n}\in \mathbb{N}^{d}\setminus\{0\}\mid  \mathbf{n}<\mathbf{x}\},
\operatorname{B}(\mathbf{x})=\{\mathbf{n}\in \mathbb{N}^{d} \mid  \mathbf{n}\le \mathbf{x}\},
\] 
where $\le$ denotes the usual partial order in $\mathbb{N}^{d}$.

\begin{proposition}
Let $H\subset \mathbb{N}^{d}\setminus \{\mathbf{0}\}$. Then $\mathbb{N}^{d}\setminus H$ is a monoid if and only if for every $\mathbf{h}\in H$, $\mathbf{h}-\mathbf{x}\in H$ for all $\mathbf{x}\in \operatorname{B}(\mathbf{h})\setminus H$.
\end{proposition}
\begin{proof}
\emph{Necessity}. Suppose there exists $\mathbf{h}\in H$ such that $\mathbf{y}=\mathbf{h}-\mathbf{x}\notin H$ for some $\mathbf{x}\in \operatorname{B}(\mathbf{h})\setminus H$. Then $\mathbf{h}=\mathbf{x}+\mathbf{y}$ with $\mathbf{x},\mathbf{y}\in \mathbb{N}^{d}\setminus H$; a contradiction.

\emph{Sufficiency}. Suppose that $\mathbb{N}^{d}\setminus H$ is not a monoid. Then there exists  $\mathbf{h}\in H$ such that $\mathbf{h}=\mathbf{x}+\mathbf{y}$ with $\mathbf{x},\mathbf{y}\in \mathbb{N}^{d}\setminus H$. In particular, $\mathbf{x}\in \operatorname{B}(\mathbf{h})\setminus H$ and $\mathbf{h}-\mathbf{x}\notin H$, contradicting the hypothesis.
\end{proof}

Observe that $\operatorname{B}(\mathbf{h})$ is a finite set for each $\mathbf{h}\in \mathbb{N}^{d}$. So if $H$ is a finite set, the condition in the previous proposition is easy to test. In particular, it is also possible to obtain a procedure to test if a given finite set $H$ is the set of gaps of a generalized numerical semigroup, without using Algorithm~\ref{alg:gfh}.

\section[Algorithm for gaps]{An algorithm to compute the set of gaps of a generalized numerical semigroup from a generating set}\label{sec:gaps}

Let $S\subseteq \mathbb{N}^{d}$ be a generalized numerical semigroup and suppose that a finite set of (not necessarily minimal) generators is known. A possible way to compute $\operatorname{H}(S)$ is via the following characterization that appears in  \cite{Analele}.

\begin{theorem}\cite[Theorem 2.8]{Analele} Let $d\geq 2$ and let $S=\langle A\rangle $ be the monoid generated by a set $A\subseteq \mathbb{N}^{d}$. Then $S$ is a generalized numerical semigroup if and only if the set $A$ fulfills the following conditions.

\begin{enumerate}

\item For all $j\in\{1,2,\ldots,d\}$, there exist $a_{1}^{(j)},a_{2}^{(j)},\ldots,a_{r_{j}}^{(j)}$, $r_{j}\in \mathbb{N}\setminus \{0\}$, such that $\gcd(a_{1}^{(j)},a_{2}^{(j)},\ldots,a_{r_{j}}^{(j)})=1$, and  $a_{1}^{(j)}\mathbf{e}_{j},a_{2}^{(j)}\mathbf{e}_{j},\ldots,a_{r_{j}}^{(j)}\mathbf{e}_{j}\in A$. In particular, if $S_i$ is the intersection of $S$ with the $i$th coordinate, then $S_i$ is a numerical semigroup. 

\item For every $i,k\in\{1,\ldots,d\}$, with $i<k$, there exist $\mathbf{x}_{ik},\mathbf{x}_{ki}\in A$ such that $\mathbf{x}_{ik}=\mathbf{e}_{i}+n_{i}^{(k)}\mathbf{e}_{k}$ and $\mathbf{x}_{ki}=\mathbf{e}_{k}+n_{k}^{(i)}\mathbf{e}_{i}$ for some $n_{i}^{(k)},n_{k}^{(i)}\in \mathbb{N}$.
\end{enumerate}

Furthermore, %let $S_{j}$ be the numerical semigroup generated by $\{a_{1}^{(j)},a_{2}^{(j)},\ldots,a_{r_{j}}^{(j)}\}$ and $F^{(j)}$ the Frobenius number of $S_{j}$, for $j=1,\ldots,d$. Let 
if we define $\mathbf{v}=(v^{(1)},v^{(2)},\ldots,v^{(d)})\in \mathbb{N}^{d}$  as 
%\[ v^{(j)}=\sum_{i\neq j}^{d}\operatorname{F}(S_i)n_{i}^{(j)}+\operatorname{F}(S_j).\]
\[ v^{(j)}=\sum_{i\neq j}^{d}\operatorname{F}^{(i)}n_{i}^{(j)}+\operatorname{F}^{(j)},\]
where $\operatorname{F}^{(i)}=\max\{\operatorname{F}(S_i),0\}$  for all $i$, then $\operatorname{H}(S)\subseteq \operatorname{B}(\mathbf{v})$.%\{\mathbf{n}\in \mathbb{N}^{d}\mid \mathbf{n}\leq \mathbf{v}\}$.

\label{nd}\end{theorem}

%Note that the first condition in the characterization provided by the above result implies that the intersection of a generalized numerical semigroup with any coordinate axis is a numerical semigroup. 

\begin{algorithm}[H]
\caption{Algorithm to compute $\operatorname{H}(S)$ from a finite set $A$ such that $\langle A\rangle=S$ \label{alg:gaps-from-generators}}
\DontPrintSemicolon
\KwData{A finite set $A\subseteq \mathbb{N}^{d}$}
\KwResult{If $S:=\langle A\rangle$ is a generalized numerical semigroup, $\operatorname{H}(S)$ is computed}
\nl \For{$i\in \{1,\ldots,d\}$}{
           Gather in a set $A_{j}$ the elements $a\in \mathbb{N}$ such $a\mathbf{e}_{j}\in A$ .\;
            \eIf{$S_{j}=\langle A_{j}\rangle$ is a numerical semigroup}             {Compute $\operatorname{F}^{(i)}=\max\{\operatorname{F}(S_j),0\}$\;} 
{$\langle A \rangle$ is not a generalized numerical semigroup\; STOP}
            
            }
\nl \For{$i\in\{1,\ldots,d\}$}{
             \For{$k\in \{1,\ldots,d\}$, $k\neq i$}{
                       Find $n_{i}^{(k)}\in \mathbb{N}$ such that $\mathbf{e}_{i}+n_{i}^{(k)}\mathbf{e}_{k}\in A$\;
             \If{$n_{i}^{(k)}$ does not exist}{$\langle A\rangle$ is not a generalized numerical semigroup\; STOP}
                                        }
                        }
\nl \For{$i\in \{1,\ldots,d\}$}{
           $v^{(j)}=\sum_{i\neq j}^{d}F^{(i)}n_{i}^{(j)}+F^{(j)}$
           }

$\mathbf{v}=(v^{(1)},v^{(2)},\ldots,v^{(d)})$\;
$\operatorname{B}(\mathbf{v})=\{\mathbf{n}\in \mathbb{N}^{d}\mid \mathbf{n}\leq \mathbf{v}\}$\; 
$\operatorname{H}(S)=\emptyset$\;
\nl \For{$\mathbf{x}\in \operatorname{B}(\mathbf{v})$}{
                 \If{$\mathbf{x}\notin \langle A\rangle$}{
                 $\operatorname{H}(S)=\operatorname{H}(S)\cup \{\mathbf{x}\}$}
                                      }
                                  
           \Return $\operatorname{H}(S)$
\end{algorithm}

A brief description of main parts of the algorithm follows.
\begin{itemize}
\item Line 1: we check if the first condition of Theorem~\ref{nd} holds; if not, $A$ does not generate a generalized numerical semigroup.
\item Line 2: we check if the second condition of Theorem~\ref{nd} is satisfied, and if this is not the case, $A$ does not generate a generalized numerical semigroup.
\item Line 3: we compute all coordinates of the vector $\mathbf{v}$ referred in Theorem~\ref{nd} such that $\operatorname{H}(S)\subseteq \operatorname{B}(\mathbf{v})$.
\item Line 4: observe that $\operatorname{B}(\mathbf{v})$ is a finite set, so we have to look for the gaps of the semigroup among a finite number of elements. 
\end{itemize}

\section{An alternative procedure to compute all generalized numerical semigroups of given genus} \label{sec:alt-alg}

The algorithm described in Section~\ref{sec:first-alg} allows to obtain all generalized numerical semigroups in $\mathbb{N}^{d}$ up to genus $g$, and in particular, to compute $N_{g,d}$, the cardinality of $\mathcal{S}_{g,d}$. Observe that with that procedure, in order to compute $N_{g,d}$ one has to compute all $N_{g',d}$ for all $g'<g$, that is, the whole semigroup tree $\mathcal{T}_\preceq$ up to level~$g$. 

In this section we define a rooted tree that contains all generalized numerical semigroups in $\mathbb{N}^{d}$ of fixed genus $g$, so that we are able to compute the set $\mathcal{S}_{g,d}$ without considering all generalized numerical semigroups of genus less than $g$. Such rooted tree is inspired by the \emph{ordinarization tranform} for numerical semigroups, introduced in \cite{bras2012ordinarization}.

%\begin{definition} indeed this is not true if \preceq is lex
%Let $\preceq$ be a relaxed monomial order in $\mathbb{N}^{d}$ and $\mathbf{s}\in \mathbb{N}^{d}$. The set $S=\{\mathbf{x}\in \mathbb{N}^{d}\mid \mathbf{s}\preceq \mathbf{x}\}\cup \{\mathbf{0}\}$ is a generalized numerical semigroup, and we call it an \emph{ordinary} generalized numerical semigroup with respect to $\preceq$.\\  

\begin{proposition}
Let $\preceq$ be a relaxed monomial order in $\mathbb{N}^{d}$ and $\mathbf{s}\in \mathbb{N}^{d}$. Suppose that the set $\{\mathbf{t}\in \mathbb{N}^{d}\mid \mathbf{t}\preceq \mathbf{s}\}$ is finite. Then the set $S=\{\mathbf{x}\in \mathbb{N}^{d}\mid \mathbf{s}\prec \mathbf{x}\}\cup \{\mathbf{0}\}$ is a generalized numerical semigroup.
\end{proposition}
\begin{proof}
By hypothesis, $\mathbb{N}^{d}\setminus S$ is finite. Moreover if $\mathbf{x},\mathbf{y}\in S$ then $\mathbf{s}\preceq\mathbf{x}$ and $\mathbf{s}\preceq\mathbf{y}$, so  $\mathbf{s}\preceq\mathbf{x}+\mathbf{y}$ since $\preceq$ is a relaxed monomial order.
\end{proof}
\begin{definition}
Let $\preceq$ be a relaxed monomial order. A monoid $S\subseteq \mathbb{N}^{d}$ satisfying the hypothesis of the above proposition is an \emph{ordinary} generalized numerical semigroup with respect to $\preceq$.

Let $\{\mathbf{0}=\mathbf{s}_{0}\preceq \mathbf{s}_{1}\preceq \cdots \preceq \mathbf{s}_{g}\}$ be the list of the first $g+1$ elements in $\mathbb{N}^{d}$, ordered by $\preceq$. We denote the ordinary generalized numerical semigroup in $\mathbb{N}^{d}$ of genus $g$ with respect to $\preceq$ by \[R_{g,d}(\preceq)=\{\mathbf{x}\in \mathbb{N}^{d}\mid \mathbf{s}_{g}\prec \mathbf{x}\}\cup \{\mathbf{0}\}.\]
\end{definition}

Observe that the previous definition of ordinary generalized numerical semigroup depends strongly on the relaxed monomial order defined: different relaxed monomial orders define different ordinary generalized numerical semigroup of given genus. While for numerical semigroups, once $g$ is fixed, there exists a unique ordinary numerical semigroup of genus $g$, for generalized numerical semigroups (following the above definition) the uniqueness of the ordinary generalized numerical semigroup of genus $g$ can only be guaranteed by fixing also a relaxed monomial order $\preceq$.\\

If $S\subseteq \mathbb{N}^{d}$ is a generalized numerical semigroup, we can define the set of \emph{special gaps} of $S$ as \[\operatorname{SG}(S)=\{\mathbf{h}\in \operatorname{H}(S)\mid 2\mathbf{h}\in S, \mathbf{h}+\mathbf{s}\in S\ \mbox{for all}\ \mathbf{s}\in S\setminus\{0\}\}.\]   
Observe that $\operatorname{H}(S)$ and $\operatorname{SG}(S)$ do not depend on the fixed relaxed monomial order. The following result, which is the analogue to \cite[Proposition 4.33]{rosales2009numerical} for generalized numerical semigroups, characterizes the elements in $\operatorname{SG}(S)$. It highlights the duality between the concepts of minimal generator and special gap. The proof is straightforward and we will omit it.

\begin{proposition}\label{prop:char-sg}
Let $S\subseteq \mathbb{N}^{d}$ be a generalized numerical semigroup and let $\mathbf{h}\in \operatorname{H}(S)$. Then $S\cup \{\mathbf{h}\}$ is a generalized numerical semigroup if and only if $\mathbf{h}\in \operatorname{SG}(S)$.
\end{proposition}

The following result will help us to construct a generalized numerical semigroup from another  by removing the multiplicity and adding the Frobenius element, and thus keeping the genus untouched.

\begin{lemma}
Let $S\subsetneq \mathbb{N}^{d}$ be a generalized numerical semigroup, let $\preceq$ be a relaxed monomial order and suppose that $S$ is not ordinary with respect to~$\preceq$. If $T=(S\cup \{\mathbf{F}_{\preceq}(S)\})\setminus \{\mathbf{m}_{\preceq}(S)\}$, then
\begin{itemize}
    \item $T$ is a generalized numerical semigroup,
    \item $\mathbf{m}_{\preceq}(S)\in \operatorname{SG}(T)$, $\mathbf{m}_{\preceq}(S)\prec \mathbf{m}_{\preceq}(T)$, and 
    \item  %$\mathbf{F}_{\preceq}(S)\succ \mathbf{F}_{\preceq}(T)$, and  
    $\mathbf{F}_{\preceq}(S)\in \operatorname{U}_\preceq(T\cup\{\mathbf{m}_\preceq(S)\})\setminus\{\mathbf{m}_\preceq(S)\}$.
\end{itemize}
\label{treeO1d}
\end{lemma}
\begin{proof}
From the definition of relaxed monomial order, $\mathbf{F}_\preceq(S)\in \operatorname{SG}(S)$. Hence $S\cup \{\mathbf{F}_{\preceq}(S)\}$ is a generalized numerical semigroup in light of Proposition~\ref{prop:char-sg}. Since $S$ is not ordinary, we have that $\mathbf{m}_\preceq(S)\preceq \mathbf{F}_\preceq(S)$. In fact, if $S$ is not ordinary, there exist $\mathbf{s}\in S$ and $\mathbf{h}\in H(S)$ such that  $\mathbf{s}\prec \mathbf{h}$, and in particular $\mathbf{m}_\preceq(S)\preceq \mathbf{s}\prec \mathbf{h}\preceq \mathbf{F}_{\preceq}(S)$. It follows that $\mathbf{m}_\preceq(S)=\mathbf{m}_\preceq(S\cup\{\mathbf{F}_\preceq(S)\})$. Thus,
by Corollary~\ref{RemoveMult}, $T$ is a generalized numerical semigroup. 

Also, as $T\cup \{\mathbf{m}_{\preceq}(S)\}=S \cup \{\mathbf{F}_{\preceq}(S)\}$ is a generalized numerical semigroup, we have $\mathbf{m}_{\preceq}(S)\prec \mathbf{m}_{\preceq}(T)$ and, by Proposition~\ref{prop:char-sg}, $\mathbf{m}_{\preceq}(S)\in \operatorname{SG}(T)$. Moreover, $\mathbf{F}_{\preceq}(S)$ is a minimal generator of $T\cup \{\mathbf{m}_{\preceq}(S)\}$ and $\mathbf{m}_\preceq(S)\neq \mathbf{F}_{\preceq}(S)\succ \mathbf{F}_{\preceq}(T\cup\{\mathbf{m}_\preceq(S)\})$. 
\end{proof}

% \begin{proposition}
% Let $S\subseteq \mathbb{N}^{d}$ be a generalized numerical semigroup, $\preceq$ a relaxed monomial order in $\mathbb{N}^{d}$ and suppose that $S$ is not ordinary with respect to $\preceq$. Then $T=(S\cup \{\mathbf{F}_{\preceq}(S)\})\setminus \{\mathbf{m}_{\preceq}(S)\}$ is a generalized numerical semigroup with $\mathbf{m}_{\preceq}(S)\preceq \mathbf{m}_{\preceq}(T)$ and $\mathbf{F}_{\preceq}(S)\succeq\mathbf{F}_{\preceq}(T)$.
% \label{ordprec}  
% \end{proposition}
% \begin{proof}
% It follows by Corollary~\ref{RemoveMult} by taking into account that under the standing hypothesis, $\mathbf{m}_{\preceq}(S)\preceq \mathbf{F}_{\preceq}(S)$, so $\mathbf{m}_{\preceq}(S\cup \{\mathbf{F}_{\preceq}(S)\})=\mathbf{m}_{\preceq}(S)$. The second statement is trivial.
% \end{proof}

\begin{definition}
Recall that $\mathcal{S}_{g,d}$ is the set of all generalized numerical semigroups in $\mathbb{N}^{d}$ of genus $g$. We define the \emph{ordinarization transform} with respect to $\preceq$,
$\mathcal{O}_\preceq:\mathcal{S}_{g,d}\rightarrow \mathcal{S}_{g,d}$ as follows:
\[\mathcal{O}_\preceq(S)=
\begin{cases}
(S\cup \{\mathbf{F}_{\preceq}(S)\})\setminus \{\mathbf{m}_{\preceq}(S)\}, & \hbox{if }S\hbox{is not ordinary},\\
S,&\hbox{otherwise}.
\end{cases}\]

Note that this map is well defined by Lemma~\ref{treeO1d}.

%Note that if $S\in\mathcal{S}_{g,d}$, then $\mathcal{O}_\preceq(S)\in \mathcal{S}_{g,d}$. 
\end{definition}

%Moreover, if $S\in \mathcal{S}_{g,d}$, then, by Lemma~\ref{treeO1d}, there exists $n\in \mathbb{N}$ such that $\mathcal{O}_\preceq^{n}(S)=R_{g,d}(\preceq)$.

%If $S\in \mathcal{S}_{d}$ and $S\neq R_{g,d}(\preceq)$, then, by Lemma~\ref{treeO1d}, there exists $n\in \mathbb{N}$ such that $\mathcal{O}_\preceq^{n}(S)=R_{g,d}(\preceq)$. Moreover if $S\in\mathcal{S}_{g,d}$, then $\mathcal{O}_\preceq(S)\in \mathcal{S}_{g,d}$.

\begin{lemma} Let $T\subseteq \mathbb{N}^{d}$ be a generalized numerical semigroup and  $\preceq$ a relaxed monomial order in $\mathbb{N}^{d}$. If there exist $\mathbf{h}\in \operatorname{SG}(T)$ with $\mathbf{h}\prec \mathbf{m}_{\preceq}(T)$ and $\mathbf{x}\in \operatorname{U}_\preceq(T\cup \{\mathbf{h}\})\setminus\{\mathbf{h}\}$, then $\mathcal{O}_\preceq((T\cup \{\mathbf{h}\})\setminus \{\mathbf{x}\})=T$.
\label{treeO2d}
\end{lemma}
\begin{proof}
If $S=(T\cup \{\mathbf{h}\})\setminus \{\mathbf{x}\}$, then we have $T=(S\cup \{\mathbf{x}\})\setminus \{\mathbf{h}\}$, and so in light of Lemma~\ref{treeO1d}, it suffices to prove that $\mathbf{h}=\mathbf{m}_{\preceq}(S)$ and $\mathbf{x}=\mathbf{F}_{\preceq}(S)$. Since $\mathbf{h}\preceq\mathbf{m}_{\preceq}(T)$, we deduce $\mathbf{h}\preceq \mathbf{t}$ for every $\mathbf{t}\in T$, and consequently $\mathbf{h}=\mathbf{m}_{\preceq}(S)$. Furthermore, since $\mathbf{x}\in \operatorname{U}_\preceq(T\cup \{\mathbf{m}_{\preceq}(S)\})$, we get $\mathbf{x}\succ\mathbf{F}_{\preceq}(T\cup \{\mathbf{h}\})=\mathbf{F}_{\preceq}(S\cup \{\mathbf{x}\})$, whence $\mathbf{x}=\mathbf{F}_{\preceq}(S)$.
\end{proof}

Let $g\in \mathbb{N}$, $\preceq$ be a relaxed monomial order in $\mathbb{N}^{d}$, and $\mathcal{E}$ be the set of pairs of the form $(S,\mathcal{O}_\preceq(S))$.
We denote the directed graph $(\mathcal{S}_{g,d},\mathcal{E})$ by $\mathcal{T}_{g,\preceq}^{d}$.
%\begin{definition}
%Let $g\in \mathbb{N}$ and let us fix a relaxed monomial order $\preceq$ in $\mathbb{N}^{d}$. We define the directed graph $\mathcal{T}_{g,\preceq}^{d}=(\mathcal{S}_{g,d},\mathcal{E})$, where $\mathcal{E}$ is the set of pairs of the form $(S,\mathcal{O}_\preceq(S))$. 
%\end{definition}

\begin{theorem}
Let $g\in \mathbb{N}$. The graph $\mathcal{T}_{g,\preceq}^{d}$ is a rooted tree with root $R_{g,d}(\preceq)$. Moreover, if $T\in \mathcal{S}_{g,d}$, then the sons of $T$ are the semigroups of the form $(T \cup \{\mathbf{h}\})\setminus \{\mathbf{x}\}$ with $\mathbf{h}\in \operatorname{SG}(T)$, $\mathbf{h}\prec \mathbf{m}_\preceq(T)$, and $\mathbf{x}\in \operatorname{U}_\preceq(T \cup \{\mathbf{h}\})\setminus\{ \mathbf{h}\}$.
%obtained considering first all $\mathbf{h}\in \operatorname{SG}(T)$ with $\mathbf{h}\preceq \mathbf{m}_{\preceq}(T)$ and after, for all semigroups $T_{\mathbf{h}}=T\cup \{\mathbf{h}\}$, considering $T_{\mathbf{h}}\setminus \{\mathbf{x}\}$ for all $\mathbf{x}\in \operatorname{U}_\preceq(T_{\mathbf{h}})$ and $\mathbf{x}\neq \mathbf{m}_{\preceq}(T_{\mathbf{h}})$.
\label{generagd}
\end{theorem}
\begin{proof}
Let $T\in \mathcal{S}_{g,d}$, we define the following sequence: 
\begin{itemize}
\item $T_{0}=T$,
\item $T_{i+1}=\left\lbrace\begin{array}{ll}
         \mathcal{O}_\preceq(T_{i}) & \mbox{if }T_{i}\neq R_{g,d}(\preceq),  \\
         R_{g,d}(\preceq) & \mbox{otherwise,}
        \end{array}
        \right.$
\end{itemize}
in particular $T_{i}=\mathcal{O}_\preceq^{i}(T)$ for all $i$. 
%We know that there exists a nonnegative integer $k$ such that $T_{k}=\mathcal{O}_\preceq^{k}(T)=R_{g,d}(\preceq)$. 
This sequence stabilizes at a certain point, since the multiplicity must be among the first $g+1$ elements of $\mathbb{N}^d$ (ordered by $\preceq$), and at each step the multiplicity increases. Let $k$ be the first integer such that $T_{k}=R_{g,d}(\preceq)$.
So the edges $(T_{0},T_{1}),(T_{1},T_{2}),\ldots,(T_{k-1},T_{k})$ provide the unique path from $T$ to the ordinary numerical semigroup of genus $g$ (with respect to $\preceq$).

By Lemma~\ref{treeO2d}, every pair $(T \cup \{\mathbf{h}\})\setminus \{\mathbf{x}\},T)$ is an edge of $\mathcal{T}_{g,\preceq}^{d}$ for every choice of $\mathbf{h}$ and $\mathbf{x}$ as in the statement, that is $(T \cup \{\mathbf{h}\})\setminus \{\mathbf{x}\},T)$ is a son of $T$. Conversely let $(S,T)$ be an edge of $\mathcal{T}_{g,\preceq}^{d}$. From the definition we have $T= \mathcal{O}_\preceq(S)$ so $T=(S\cup \{\mathbf{F}_{\preceq}(S)\})\setminus \{\mathbf{m}_{\preceq}(S)\}$. In particular $S=(T\cup \{\mathbf{m}_\preceq(S)\}\setminus \{\mathbf{F}_\preceq(S)\}$ and from Lemma~\ref{treeO1d} we have $\mathbf{m}_{\preceq}(S)\in \operatorname{SG}(T)$ with $\mathbf{m}_{\preceq}(S)\prec \mathbf{m}_{\preceq}(T)$, and $\mathbf{F}_{\preceq}(S)\in \operatorname{U}_\preceq(T\cup\{\mathbf{m}_\preceq(S)\})\setminus\{\mathbf{m}_\preceq(S)\}$. This fact ensures that all descendants in $\mathcal{T}_{g,\preceq}^{d}$ of a generalized numerical semigroup are of the described form.
\end{proof}

\begin{corollary}
Let $T\in \mathcal{S}_{g,d}$. Then $T$ is a leaf in $\mathcal{T}_{g,\preceq}^{d}$ if and only if for all $\mathbf{h}\in \operatorname{SG}(T)$ either $\mathbf{h}\succ \mathbf{m}_{\preceq}(T)$ or $\operatorname{U}_\preceq(T \cup \{\mathbf{h}\})\subseteq\{\mathbf{h}\}$.
\end{corollary}

Theorem~\ref{generagd} allow us to write another algorithm to produce all generalized numerical semigroups in $\mathbb{N}^{d}$ of genus $g$; the previous corollary gives us the stop condition.

\begin{algorithm}[H]  \label{alg1}
\caption{Algorithm for computing the set $\mathcal{S}_{g,d}$ \label{alg:S-g-d}}
\DontPrintSemicolon

\KwData{Two integers $g,d\in \mathbb{N}$ and a relaxed monomial order $\preceq$ in \(\mathbb{N}^d\).}
\KwResult{$\mathcal{S}_{g,d}$}

 Compute $R_{g,d}(\preceq)$.\; 
 $\mathcal{S}_{g,d}=\{R_{g,d}(\preceq)\}, \mathcal{L}=\{R_{g,d}(\preceq)\}$.\;
\nl \While{Exists $S\in \mathcal{L}$ \emph{such that} \(\{\mathbf{h}\in \operatorname{SG}(S)\mid \mathbf{h}\preceq \mathbf{m}_{\preceq}(S)\}\neq \emptyset\)}{
   $\mathcal{I}=\emptyset$\;
   \For{$S\in \mathcal{L}$}
                {
       \If{\(\{\mathbf{h}\in \operatorname{SG}(S)\mid \mathbf{h}\preceq \mathbf{m}_{\preceq}(S)\}\neq \emptyset\)}{
          $\mathcal{R}=\emptyset$\;
      \nl   \For{$\mathbf{h}\in \{\mathbf{h}\in \operatorname{SG}(S)\mid \mathbf{h}\preceq \mathbf{m}_{\preceq}(S)\}$}{
             $\mathcal{R}=\mathcal{R}\cup \{S\cup\{\mathbf{h}\}\}$
           } 
         \nl    \For{$T\in \mathcal{R}$}{
                  \For{$\mathbf{x}\in \operatorname{U}_\preceq(T)$ with $\mathbf{x}\neq \mathbf{m}_{\preceq}(S)$}{
                  $\mathcal{I}=\mathcal{I}\cup \{T\setminus\{x\}\}$\;
                  %$\mathcal{S}_{g,d}=\mathcal{S}_{g,d}\cup L$
                      }
                      }
                   }
                 }
                 $\mathcal{S}_{g,d}=\mathcal{S}_{g,d}\cup \mathcal{I}$\;
               \nl  $\mathcal{L}=\mathcal{I}$\;
                }
 \Return $\mathcal{S}_{g,d}$\;  
\end{algorithm}

Let us give a few comments regarding the labeled lines of Algorithm \ref{alg1}.
\begin{itemize}
\item Line 1: the algorithm stops when all computed semigroups are leaves of $\mathcal{T}_{g,\preceq}^{d}$; we gather in $\mathcal{L}$ the semigroups for which we have to do computations at each step.
\item Line 2: for each semigroup $S$ of the current step we have to compute its sons, and for this we need all semigroups $S\cup \{\mathbf{h}\}$ with $\mathbf{h}\in \operatorname{SG}(S)$ and smaller than $\mathbf{m}_{\preceq}(S)$ with respect to $\preceq$.
\item Line 3: for each semigroup $T$ computed in the previous line we compute $T\setminus \{\mathbf{x}\}$ for all $\mathbf{x}\in \operatorname{U}_\preceq(T)$. These are the sons of all semigroup in $\mathcal{L}$, that we collect in $\mathcal{I}$.
\item Line 4: in the next step we have to repeat the same procedure considering the semigroups in $\mathcal{I}$.
\end{itemize}

\begin{example}%\rm
Let $\preceq$ be the lexicographic order in $\mathbb{N}^{2}$. Consider the semigroup $R_{3,2}(\preceq)=\mathbb{N}^{2}\setminus \{(0,1),(0,2),(0,3)\}$, we compute its sons in $\mathcal{T}_{3,\preceq}^{2}$. \\
Following the pseudocode we start with $\mathcal{S}_{3,2}=\{R_{3,2}(\preceq)\}, \mathcal{L}=\{R_{3,2}(\preceq)\}$.\\
%The set of special gaps less than $\mathbf{m}_{\preceq}(R_{3,2}(\preceq))=(0,4)$ is $\{(0,2),(0,3)\}$. So we consider:
In line 1 we have $\mathbf{m}_{\preceq}(R_{3,2}(\preceq))=(0,4)$ so $\{\textbf{h}\in\operatorname{SG}(R_{3,2}(\preceq))\mid \textbf{h}\preceq \mathbf{m}_{\preceq}(0,4))\}=\{(0,2),(0,3)\}$. So, following line 2 of the pseudocode, we consider:
\begin{itemize}
\item $T_{(0,2)}=R_{3,2}(\preceq)\cup \{(0,2)\}=\mathbb{N}^{2}\setminus \{(0,1),(0,3)\}$, with $\operatorname{U}_\preceq(T_{(0,2)})=\{(1,0), (1,1),(0,5)\}$,  
\item $T_{(0,3)}=R_{3,2}(\preceq)\cup \{(0,3)\}=\mathbb{N}^{2}\setminus \{(0,1),(0,2)\}$, with $\operatorname{U}_\preceq(T_{(0,3)})=\{(1,0),(1,1),(0,3),(1,2),(0,4),(0,5)\}$  
\end{itemize}
This means that we obtain $\mathcal{R}=\{T_{(0,2)},T_{(0,3)}\}$.
So the sons of $R_{3,2}(\preceq)$, obtained with the procedure in line 3, are the following:
\begin{itemize}
\item $S_{1}=T_{(0,2)}\setminus \{(1,0)\}=\mathbb{N}^{2}\setminus \{(0,1),(0,3),(1,0)\}$,
\item $S_{2}=T_{(0,2)}\setminus \{(1,1)\}=\mathbb{N}^{2}\setminus \{(0,1),(0,3),(1,1)\}$,
\item $S_{3}=T_{(0,2)}\setminus \{(0,5)\}=\mathbb{N}^{2}\setminus \{(0,1),(0,3),(0,5)\}$,
\item $S_{4}=T_{(0,3)}\setminus \{(1,0)\}=\mathbb{N}^{2}\setminus \{(0,1),(0,2),(1,0)\}$,
\item $S_{5}=T_{(0,3)}\setminus \{(1,1)\}=\mathbb{N}^{2}\setminus \{(0,1),(0,2),(1,1)\}$,
\item $S_{6}=T_{(0,3)}\setminus \{(1,2)\}=\mathbb{N}^{2}\setminus \{(0,1),(0,2),(1,2)\}$,
\item $S_{7}=T_{(0,3)}\setminus \{(0,4)\}=\mathbb{N}^{2}\setminus \{(0,1),(0,2),(0,4)\}$,
\item $S_{8}=T_{(0,2)}\setminus \{(0,5)\}=\mathbb{N}^{2}\setminus \{(0,1),(0,2),(0,5)\}$.
\end{itemize}
In particular $\mathcal{I}=\{S_{1}S_{2},S_{3},S_{4},S_{5},S_{6},S_{7},S_{8}\}$ and $\mathcal{S}_{3,2}=\{R_{3,2}(\preceq)\}\cup \mathcal{I}$.\\
At this point (that is line 4) $\mathcal{L}=\mathcal{I}$ and we start again the procedure considering all semigroups in $\mathcal{L}$. If we continue, we obtain all generalized numerical semigroup of genus $g$.\\
Observe that for $S_{1},S_{2},S_{3},S_{6},S_{7},S_{8}$ the set $\{\mathbf{h}\in \operatorname{SG}(S)\mid \mathbf{h}\preceq \mathbf{m}_{\preceq}(S)\}$ is empty, so they are leaves in $\mathcal{T}_{3,\preceq}^{2}$ and in this second step we have to consider the procedures in line 2 and line 3 only for the semigroups $S_{4}$ and $S_{5}$.

The sons of $S_{4}$ are:
\begin{itemize}
\item $S_{9}=S_{4}\cup \{(0,2)\}\setminus \{(1,1)\}=\mathbb{N}^{2}\setminus \{(0,1),(1,0),(1,1)\}$,
\item $S_{10}=S_{4}\cup \{(0,2)\}\setminus \{(2,1)\}=\mathbb{N}^{2}\setminus \{(0,1),(1,0),(2,1)\}$,
\item $S_{11}=S_{4}\cup \{(0,2)\}\setminus \{(1,2)\}=\mathbb{N}^{2}\setminus \{(0,1),(1,0),(1,2)\}$,
\item $S_{12}=S_{4}\cup \{(0,2)\}\setminus \{(2,0)\}=\mathbb{N}^{2}\setminus \{(0,1),(1,0),(2,0)\}$,
\item $S_{13}=S_{4}\cup \{(0,2)\}\setminus \{(1,1)\}=\mathbb{N}^{2}\setminus \{(0,1),(1,0),(3,0)\}$.
\end{itemize}

The sons of $S_{5}$ are:
\begin{itemize}
\item $S_{14}=S_{5}\cup \{(0,2)\}\setminus \{(2,1)\}=\mathbb{N}^{2}\setminus \{(0,1),(1,1),(2,1)\}$.
\end{itemize}
In particular, in line 3 of the second step we obtain $\mathcal{I}=\{S_{9},S_{10},S_{11},S_{12},S_{13},S_{14}\}$. Now the third step starts.

The semigroups $S_{10},S_{11},S_{14}$ are leaves in $\mathcal{T}_{3,\preceq}^{2}$, so for the successive computations we have to consider $S_{9},S_{12},S_{13}$.

The sons of $S_{9}$ are:
\begin{itemize}
\item $S_{15}=S_{9}\cup \{(0,1)\}\setminus \{(2,0)\}=\mathbb{N}^{2}\setminus \{(1,0),(1,1),(2,0)\}$,
\item $S_{16}=S_{9}\cup \{(0,1)\}\setminus \{(3,0)\}=\mathbb{N}^{2}\setminus \{(1,0),(1,1),(3,0)\}$,
\item $S_{17}=S_{9}\cup \{(0,1)\}\setminus \{(1,2)\}=\mathbb{N}^{2}\setminus \{(1,0),(1,1),(1,2)\}$.
\end{itemize}

The sons of $S_{12}$ are:
\begin{itemize}
\item $S_{18}=S_{12}\cup \{(0,1)\}\setminus \{(2,1)\}=\mathbb{N}^{2}\setminus \{(1,0),(2,0),(2,1)\}$,
\item $S_{19}=S_{12}\cup \{(0,1)\}\setminus \{(3,0)\}=\mathbb{N}^{2}\setminus \{(1,0),(2,0),(3,0)\}$,
\item $S_{20}=S_{12}\cup \{(0,1)\}\setminus \{(4,0)\}=\mathbb{N}^{2}\setminus \{(1,0),(2,0),(4,0)\}$,
\item $S_{21}=S_{12}\cup \{(0,1)\}\setminus \{(5,0)\}=\mathbb{N}^{2}\setminus \{(1,0),(2,0),(5,0)\}$.
\end{itemize}

The only son of $S_{13}$ is:
\begin{itemize}
\item $S_{22}=S_{13}\cup \{(0,1)\}\setminus \{(5,0)\}=\mathbb{N}^{2}\setminus \{(1,0),(3,0),(5,0)\}$.
\end{itemize}

At the end of third step we have $\mathcal{L}=\{S_{15},S_{16},S_{17},S_{18},S_{19},S_{20},S_{21},S_{22}\}$. For each one of these semigroups $\mathbf{m}_{\preceq}(S)=(0,1)$ and the set $\{\mathbf{h}\in \operatorname{SG}(S)\mid \mathbf{h}\preceq \mathbf{m}_{\preceq}(S)\}$ is empty, so the procedure ends. So we have obtained exactly 23 different generalized numerical semigroups. We can see a representation of the graph $\mathcal{T}_{3,\preceq}^{2}$ in Figure~\ref{ordinary-tree}.
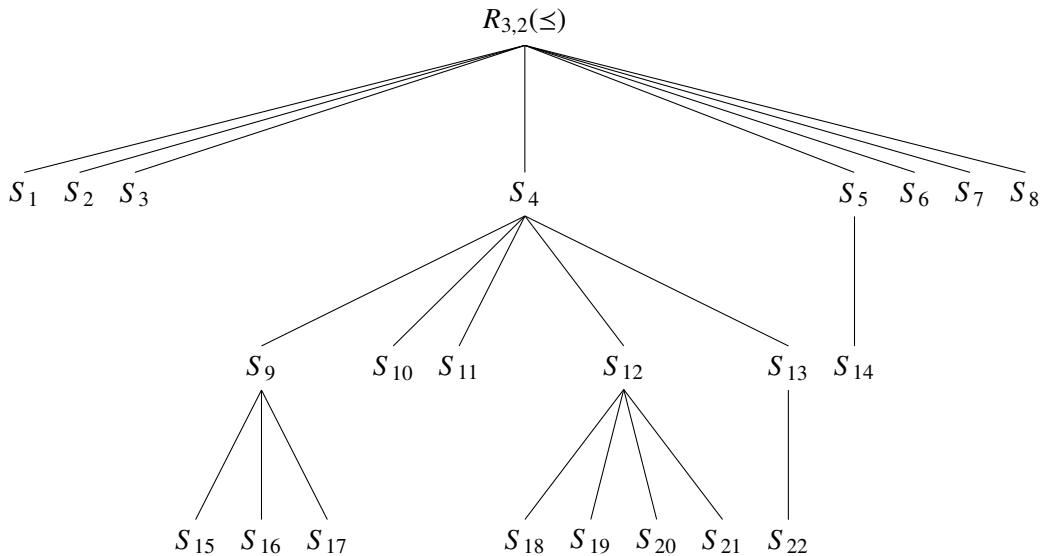
\begin{figure}
\begin{tikzpicture}
\tikzset{level distance=6em}
\Tree
	[.$R_{3,2}(\preceq)$  
		$S_1$ 
		$S_2$ 
		$S_3$
		[.$S_4$ 
			[.$S_9$ $S_{15}$ $S_{16}$ $S_{17}$	] 
			$S_{10}$
			$S_{11}$
			[.$S_{12}$ $S_{18}$ $S_{19}$ $S_{20}$ $S_{21}$ ]
			[.$S_{13}$ $S_{22}$ ]
		] 
		[.$S_5$ $S_{14}$ ] 		
		$S_6$
		$S_7$
		$S_8$
	]
\end{tikzpicture}
\caption{The tree $\mathcal{T}_{3,\preceq}^{2}$, with $\preceq$ the lexicographic order.}
\label{ordinary-tree}
\end{figure}

\end{example}

By Theorem~\ref{generagd} it is possible to produce all generalized numerical semigroups of genus $g$ in $\mathbb{N}^{d}$ starting from $R_{g,d}(\preceq)$. This procedure works as in the previous example, and it avoids  considering all generalized numerical semigroup of genus less then $g$, as happens in the algorithm given in Section~\ref{sec:first-alg}.
%like in the standard algorithm described for instance in \cite{failla2016algorithms}.

\section{Some remarks concerning implementations}\label{sec:implementations}

%Include examples, with times... (obtained with a laptop...)

The relaxed monomial order \texttt{GAP} used by default is the lexicographic order. Our current implementations are not yet prepared to give the user the possibility of choosing another order.
\medskip

%Details in implementations

There was an important speedup in a preliminary implementation of Algorithm~\ref{alg:gaps-from-generators}
after observing that the cartesian product of elements of the numerical semigoups in the axes consists of elements of the generalized numerical semigroup. Thus one may exclude this set from the set of possible gaps. Note one can use the relatively efficient function that the \texttt{numericalsgps} package provides for computing the set of small elements of a numerical semigroup.
\medskip

%\begin{algorithm}%[H]  
%\NoCaptionOfAlgo
%\caption{Recursive procedure for computing $N_{g,d}$
%\label{alg:rec-N-g-d}}
%\Fn(\tcc*[h]{algorithm as a recursive function}){\FRecurs{some args}}{
%\DontPrintSemicolon
%\KwData{Two integers $g,d\in \mathbb{N}$ and a relaxed monomial order $\preceq$ in \(\mathbb{N}^d\).}
%\KwResult{$N_{g,d}$}
%\textbf{procedure}\ RecusiveAffineSons($S,N$) \;
 % $\operatorname{SmSG}=\{\mathbf{h}\in \operatorname{SG}(S)\mid %\mathbf{h}\preceq \mathbf{m}_{\preceq}(S)\}$\; 
  %\If{$\operatorname{SmSG(S)}\neq \emptyset$}{
    %$\operatorname{sons}=\{(S\cup \{\mathbf{h}\})\setminus \{\mathbf{x}\}\mid \mathbf{h}\in \operatorname{SmSG},\ \mathbf{x}\in \operatorname{U}_\preceq(T \cup \{\mathbf{h}\})\setminus\{ \mathbf{h}\}\}$.
%\nl    $\operatorname{Sons}=\emptyset$\;
    %$\mathcal{R}=\emptyset$\;
     %   \For{$\mathbf{h}\in \operatorname{SmSG}$}{
      %       $\mathcal{R}=\mathcal{R}\cup \{S\cup\{\mathbf{h}\}\}$
       %    } 
        %    \For{$T\in \mathcal{R}$}{
         %         \For{$\mathbf{x}\in \operatorname{U}_\preceq(T)$ with $\mathbf{x}\neq \mathbf{m}_{\preceq}(S)$}{
%\nl                  %$\operatorname{Sons}=\operatorname{Sons}\cup %\{T\setminus\{x\}\}$\;    
 %  }     
  %   }
%\For{$T\in \operatorname{Sons}$}{     
%$E=\operatorname{RecursiveAffineSons}(T,N)$\; 
%$N=E$    
%}
%     }
 % \Return $N+1$\;
 % }
%\end{algorithm}

We implemented Algorithm~\ref{alg:S-g-d} in a recursive way. It explores the tree in a depth first manner. Note that an exploration in a breadth first manner soon causes memory problems since at least the semigroups of previous genus have to be stored, even when one is just concerned with counting.
By doing this in a depth first manner there is only a small amount of information that needs to be stored (unless we go very deep in the tree), the disadvantage being that one has to decide in advance the genus  one wants to attend. One somehow overcomes this by making some previsions on the time that will be spent based on the time taken for computations of lower genus.   
Furthermore, our implementation includes the computation of the coefficients needed for the polynomial referred in Theorem~\ref{polinomi}. In particular we considered the following GAP code: 

\lstset{language=GAP}
\begin{lstlisting}[frame=lines]
recursiveAffineSons:= function(s,N,L)
    local  H, gens, ml, smallgaps, smallPFs, SmSG, sons, t, E, i;
    
    H:=Gaps(s);
    gens:=Generators(s);
    ml:=Minimum(gens); #multiplicity
    smallgaps:=Filtered(H,j->j<ml);
    smallPFs:=Filtered(smallgaps,g->not ForAny(gens,n -> n+g in H));
    SmSG:=Filtered(smallPFs,g->not(2*g in H));
    #special gaps smaller than the multiplicity
    if not(IsEmpty(SmSG)) then
        sons:=affineSons(s,SmSG);
        for t in sons do
        E:=recursiveAffineSons(t,N,L);
        N:=E[1];
        L:=E[2];
        od;
    fi;
    N:=N+1;
    i:=Rank(H);
    L[i]:=L[i]+1;
    return [N,L];
end;
\end{lstlisting}
We used the previous recursive procedure as a local function, initializing $S:=R_{g,d}(\preceq)$ (while $\preceq$ is the lexicographic order), \texttt{N} to zero and \texttt{L} to be the zero list with $d$ entries. In fact, the code works in such a way that at the end of all recursion steps the variable \texttt{N} will contain the number of all generalized numerical semigroups with given genus, and the $i$th element of the list \texttt{L} will contain the number of generalized numerical semigroups whose set of gaps generates a vector space of dimension $i$. These variables are updated for the first time when the current semigroup has no sons and successively each time going up in the recursion. The local function \texttt{affineSons} computes the sons of the current semigroup. This function requires as input also the set of special gaps smaller than the multiplicity of the current semigroup, in order to avoid to compute many times the set \texttt{SmSG}.

Algorithm~\ref{alg:S-g-d} can be used also to compute the number of numerical semigroups of given genus. The present implementation in the package \texttt{numericalsgps} of the function \texttt{NumericalSemigroupsWithGenus} designed for this uses the ``standard tree'' and it is faster than the implementation we have for our algorithm. This is maybe due to the choice of a more appropriate encoding, which in the case of the package is presently based on Apéry sets. Even if it can be introduced also for affine semigroups (see \cite{2019arXiv190311028G}), the Apéry set is in general an infinite set for generalized numerical semigroups, while it is finite for numerical semigroups. For this and other reasons, in what concerns the aim of Algorithm~\ref{alg:S-g-d}, at this moment we do not know if it is not possible to work with Apéry sets (or something related to it) in generalized numerical semigroups in the same way as for numerical semigroup.  

In order to have sufficient examples to test our implementations at the time this work was being prepared we implemented a method to generate a pseudo-random generalized numerical semigroup of a genus given: from a semigroup of genus \(g-1\) one obtains one of genus \(g\) by removing one minimal generator at random (obtained using the \texttt{GAP} function \texttt{RandomList}).

\section{Some computational results}\label{sec:experiments}
Let, as above, $N_{g,d}$ be the number of all generalized numerical semigroups in $\mathbb{N}^{d}$ of genus $g$. If $d=1$, it has been proved that the sequence $\{N_{g,1}\}$ has a Fibonacci-like behaviour (see \cite{zhai2013fibonacci}), that is, $\lim_{g\rightarrow \infty}N_{g,1}/N_{g-1,1}=\phi$ where $\phi$ is the golden ratio, as conjectured by M. Bras-Am\'oros in \cite{bras2008fibonacci}. That conjecture was justified by the computation of the values $N_{g,1}$ for $g=1$ up to $g=50$. A natural question is whether the sequence $\{N_{g,d}\}$ with $d>1$ has a particular behaviour. 
In order to shed some light on it we did several computations, presented in the tables in this section.
%\medskip

\begin{table}
\caption{Computational results for $N_{g,2}$}
\begin{tabular}{ccccc}
\toprule
$g$ &  $N_{g,2}$ & $N_{g-1,2}+N_{g-2,2}$ & $\frac{N_{g-1,2}+N_{g-2,2}}{N_{g,2}}$ & $\frac{N_{g,2}}{N_{g-1,2}}$ \\
\midrule
0	& 1	&   & 	&	   \\
1	& 2	&	&   &	2  \\

2 &	7	& 3	& 0,4285714286	& 3,5  \\
3 &	23	& 9  &	0,3913043478	&  3,2857142857\\
4 &	71	& 30 &	0,4225352113  & 	3,0869565217\\
5 &	210	 & 94  &  0,4476190476  &	2,9577464789\\
6 &	638	 & 281 &  0,4404388715  &	3,0380952381\\
7 &	1894 &	848  &	0,4477296727  &	2,9686520376\\
8 &	5570 &	2532 &	0,4545780969  &	2,9408658923\\
9 &	16220 &	7464 &	0,4601726264 &	2,9120287253\\
10 & 46898 & 21790 & 0,4646253572 &	2,8913686806\\
11 & 134856 & 63118 &	0,4680399834 &	2,8755170796\\
12 & 386354 &	181754 &	0,4704338508 &	2,8649374147\\
13 & 1102980 &	521210 &	0,4725470997 &	2,8548429678\\
14 & 3137592 &	1489334 &	0,4746742088 &	2,8446499483\\
15 & 8892740 &	4240572 &	0,4768577514 &	2,8342563342\\
16	& 25114649	& 12030332	& 0,4790165294	& 2,8241744389\\
17	& 70686370	& 34007389	& 0,4811024954	& 2,8145473982\\
18 & 198319427 & 95801019	& 0,4830642184	& 2,8056247194\\
19 & 554813870	& 269005797	 & 0,4848577362	 & 2,797577012\\
20	& 1548231268 & 753133297 & 0,4864475434	& 2,7905417505\\
21	& 4310814033 & 2103045138	& 0,4878533664	& 2,7843476114\\

\bottomrule
\end{tabular}
\label{tabella1}
\end{table}

\begin{remark}\rm
The values presented in Table~\ref{tabella1} for \(g\in\{19,20,21\}\) are new. The others coincide with those presented in \cite[Table~3]{garcia2018extension}, except for \(g=18\), which seems to be wrong in that paper.
Our suspicion is based on the fact that that the number \(n_{18}/n_{17}\) in the second column of the table in their paper, which corresponds to the number \(N_{18,2}/N_{17,2 }\) in the forth column of Table~\ref{tabella1}, is smaller than expected, from the regularity exhibited by the sequence of values in that column.
\end{remark}
%Some values for $d>2$ are the following:  

\medskip

%A larger number of values of $N_{g,d}$ for $d=2$ and $d=3$ was founded by \cite{garcia2018extension} and the values obtained in Table 1 and Table 2 are exactly the same of that work.\\ 
We compute some values of $N_{g,d}^{(r)}$ in order to build the polynomial $F_{g}(d)$ of Theorem~\ref{polinomi}. Some polynomials are given in \cite{failla2016algorithms}, exactly:

\begin{itemize}
\item $F_{1}(d)=d$           %1\binom{d}{1}=d$.
\item $F_{2}(d)=\frac{3}{2}d^{2}+\frac{1}{2}d$               %2\binom{d}{1}+3\binom{d}{2}=\frac{3}{2}d^{2}+\frac{1}{2}d$.
\item $F_{3}(d)=\frac{5}{3}d^{3}+\frac{5}{2}d^{2}-\frac{1}{6}d$                             %4\binom{d}{1}+15\binom{d}{2}+10\binom{d}{3}=\frac{5}{3}d^{3}+\frac{5}{2}d^{2}-\frac{1}{6}d$.
\end{itemize}

Other values of $N_{g,d}^{(r)}$ have been computed, in particular:

\[
\begin{array}{llllll}

N_{4,2}^{(2)}=57, & N_{5,2}^{(2)}=186, & N_{6,2}^{(2)}=592, & N_{7,2}^{(2)}=1816, &  N_{8,2}^{(2)}=5436, & N_{9,2}^{(2)}=15984,\\
N_{4,3}^{(3)}=100, & N_{5,3}^{(3)}=621,&  N_{6,3}^{(3)}=3230, & N_{7,3}^{(3)}=15371, &  N_{8,3}^{(3)}=69333, & N_{9,3}^{(3)}=301425,\\
N_{4,4}^{(4)}=41,  & N_{5,4}^{(4)}=672, & N_{6,4}^{(4)}=6321, & N_{7,4}^{(4)}=47432, &  N_{8,4}^{(4)}=315393, &  N_{9,4}^{(4)}=1945238, \\
 & N_{5,5}^{(5)}=196, & N_{6,5}^{(5)}=4745, & N_{7,5}^{(5)}=63205, & N_{8,5}^{(5)}=648115, & N_{9,5}^{(5)}=5742670, \\
 & & N_{6,6}^{(6)}=1057, & N_{7,6}^{(6)}=35480, & N_{8,6}^{(6)}=637312, & N_{9,6}^{(6)}=8584915, \\
 & & &  N_{7,7}^{(7)}=6322, & N_{8,7}^{(7)}= 281099, & N_{9,7}^{(7)}=6563802, \\ 
 & & & &  N_{8,8}^{(8)}=41393, & N_{9,8}^{(8)}=2355792, \\
 & & & & & N_{9,9}^{(9)}=293608.\\
\end{array}
\]

In \cite{bras2008fibonacci} one can see the following known values:
$$N_{4,1}^{(1)}=7,\ \ N_{5,1}^{(1)}=12,\ \ N_{6,1}^{(1)}=23,\ \ N_{7,1}^{(1)}=39,\ \ N_{8,1}^{(1)}=67,\ \ N_{9,1}^{(1)}=118.$$
so the following polynomials can be expressed :

\begin{itemize}
\item $F_{4}(d)=\frac{41}{24}d^{4}+\frac{77}{12}d^{3}-\frac{65}{24}d^{2}+\frac{19}{12}d$,                                 %7\binom{d}{1}+57\binom{d}{2}+100\binom{d}{3}+41\binom{d}{4}=\frac{41}{24}d^{4}+\frac{77}{12}d^{3}-\frac{65}{24}d^{2}+\frac{19}{12}d$.
\item $F_{5}(d)=\frac{49}{30}d^{5}+\frac{35}{3}d^{4}-\frac{22}{3}d^{3}+\frac{53}{6}d^{2}-\frac{14}{5}d$,                       %12\binom{d}{1}+186\binom{d}{2}+621\binom{d}{3}+672\binom{d}{4}+196\binom{d}{5}=\frac{49}{30}d^{5}+\frac{35}{3}d^{4}-\frac{22}{3}d^{3}+\frac{53}{6}d^{2}-\frac{14}{5}d$.
\item $F_{6}(d)=\frac{1057}{720}d^{6}+\frac{841}{48}d^{5}-\frac{1045}{144}d^{4}+\frac{563}{48}d^{3}+\frac{148}{45}d^{2}-\frac{15}{4}d$,

\item $F_{7}(d)=\frac{3161}{2520}d^{7}+\frac{8257}{360}d^{6}+\frac{127}{18}d^{5}-\frac{1735}{72}d^{4}+\frac{31757}{360}d^{3}-\frac{15091}{180}d^{2}+\frac{577}{21}d$,

\item $F_{8}(d)=\frac{41393}{40320}d^{8}+\frac{38921}{1440}d^{7}+\frac{128099}{2880}d^{6}-\frac{9227}{62}d^{5}+\frac{1875151}{5760}d^{4}-\frac{467041}{1440}d^{3}+\frac{1234271}{10080}d^{2}-\frac{25}{24}d$,

\item $F_{9}(d)=\frac{5243}{6480}d^{9}+\frac{14767}{504}d^{8}+\frac{58399}{540}d^{7}-\frac{203159}{720}d^{6}+\frac{301811}{540}d^{5}-\frac{24961}{144}d^{4}-\frac{3909443}{6480}d^{3}+\frac{536093}{840}d^{2}-\frac{1427}{9}d$.
\end{itemize}

\medskip

\begin{table}[h!]
\caption{Some values of $N_{g,d}$ for $d>2$.\label{tabella2}}
\begin{tabular}{ccccccccc}
\toprule
$g$ &  $N_{g,3}$ & $N_{g,4}$ & $N_{g,5}$ & $N_{g,6}$ & $N_{g,7}$ & $N_{g,8}$ & $N_{g,9}$ & $N_{g,10}$ \\
\midrule

 1 & 3  &  4  &	   5 &	  6 &   7 &	  8 &  	 9 & 10 \\
2 &  15 &  26 &	  40 &   57 &  77 &	100 &  126 & 155\\
3 &  67 & 146 &	 270 &  449 & 693 & 1012 & 1416 & 1915 \\	
4 & 292 & 811 &    1810 &	3512 &	6181 &	10122&	15681&	23245 \\
5 & 1215 &	4320&	11686 &	26538&	53361&	98096&	168336&	273522 \\
6 & 5075 &	22885&	74685 &	197960&	453922&	935426&	1775943&  3159590\\
7 & 20936 &	119968&	472430&	1461084& 3818501& 8815672& 18505065&	36024450\\
8 & 85842 &	625609 & 2973105&	10725499&	31932733&	82542263&  191448588 &	407552845\\
9 & 349731 & 3247314 &  18643540 & 78488473 & 266223972 & 770328304 & 1973498062 & 4591979390  \\
10 & 1418323  & 16800886      &       & & & & & \\
11 & 5731710  & 86739337      &       & & & & & \\
12 & 23100916 &  447283982     &       & & & & & \\
13 & 92882954 & 2304942650      &       & & & & & \\
14 & 372648740 &       &       & & & & & \\
\bottomrule
\end{tabular}
\end{table}

\begin{remark} \rm
The values of Table 2 up to dimension 9 have been effectively computed. Of course, they coincide with the values given by the polynomials $F_{g}(d)$ mentioned above, in the cases this value exists.
As a consequence of our computations we observe that we are able to compute the number of generalized numerical semigroups with genus up to 9 for any dimension (using the polynomial given by Theorem~\ref{polinomi}). The last column of the table contains the values for dimension 10.

All the values presented are new, except for those of dimension 3, up to genus 13. These already appear in \cite[Table 3]{garcia2018extension}, and our computations confirm them.

%The values of Table 2 up to genus 9 for dimension 10 was obtained by polynomials $F_{g}(d)$ mentioned above. Anyway, observe that for genus from 1 to 9 we are able to compute the number of generalized numerical semigroups of all dimension.  Moreover, for dimension less than 10 such polynomials return the same values computed and showed in Table 2.
\end{remark}

\medskip

We end the paper by doing some comparisons between $(N_{g,1})^{2}$ and  $N_{g,2}$, for $g$ not greater than $21$. Recall that, up to $g=21$ the values $N_{g,2}$ appear for the first time in this work, while the values of $N_{g,1}$ for the same values of $g$ are well known. The comparisons are registered in Table~\ref{table3}.

\begin{table}[h!]
\caption{Values to estimate $(N_{g,1})^{2}/N_{g,2}$.\label{table3}}
\begin{tabular}{ccccc}
\toprule
$g$ &  $N_{g,1}$ & $(N_{g,1})^{2}$ & $N_{g,2}$ & $\frac{(N_{g,1})^{2}}{N_{g,2}}$ \\
\midrule
1 &	1	&	1	& 2	 & 0,5\\
2 &	2	&	4	& 7 &	0,5714285714\\
3 &	4	&	16	& 23 & 0,6956521739\\
4 &	7	&	49	& 71 & 0,6901408451\\
5 &	12 & 144 &	210 &	0,6857142857\\
6 &	23 & 529 &	638	& 0,829153605\\
7 &	39 & 1521 &	1894 & 0,803062302\\
8 &	67 & 4489 &	5570 &	0,8059245961\\
9 &	118 & 13924 & 16220 &	0,8584463625\\
10 & 204 &	41616 &	46898 &	0,8873725958\\
11 & 343 &	117649 & 134856 & 0,872404639\\
12 & 592 &	350464 & 386354 & 0,9071059184\\
13	& 1001 & 1002001 & 1102980 & 0,9084489293\\
14	 & 1693 & 2866249 &	3137592 & 0,9135187112\\
15	& 2857	& 8162449 &	8892740 & 0,9178778419\\
16	& 4806 & 23097636 &	25114649 &	0,9196877886\\
17	 & 8045 & 64722025 & 70686370 &	0,9156224177\\
18	& 13467 & 181360089 & 198319427 & 0,9144847368\\
19	 & 22464 & 504631296 &	554813870 &	0,9095506138\\
20 & 37396 & 1398460816 & 1548231268 & 0,9032635142\\
21 & 62194 & 3868093636 & 4310814033 & 0,8973000474\\
\bottomrule
\end{tabular}
\end{table}

Note that for all the $g$'s in the table we have that $(N_{g,1})^{2}<N_{g,2}$. Also, note that the ratio $(N_{g,1})^{2}/N_{g,2}$ is increasing from $g=1$ to $g=16$ and starts then to be decreasing until reaching $g=21$. What happens next is not known, since $N_{g,2}$ has not been computed for $g\ge 22$.

It is easy to check using the data presented in Table~\ref{tabella2} that having decreasing ratios is no longer true when $d$ is bigger than $2$. In fact for $d\ge 3$, one quickly has $(N_{g,1})^{d}>N_{g,d}$, by letting $g$ grow. Moreover the ratio $(N_{g,1})^{d}/N_{g,d}$ seems to grow quickly, as $g$ grows.

Let us return to the case $d=2$. Some questions arise:

\begin{enumerate}
\item Is it true $(N_{g,1})^{2}<N_{g,2}$ for every $g\in \mathbb{N}$?
\item Does $\lim\limits_{g\rightarrow \infty} \frac{(N_{g,1})^{2}}{N_{g,2}}$ exist, and is nonzero?  
\end{enumerate}

If the second question has a positive answer, one has that $\lim\limits_{g\rightarrow \infty} \frac{N_{g,2}}{N_{g-1,2}}=\phi^{2}$, where $\phi$ is the golden ratio, as can easily be shown.

If the first question has a positive answer and the ratio $\frac{(N_{g,1})^{2}}{N_{g,2}}$ is decreases from a certain value on (as our experiments suggest), then it is true also that $\lim\limits_{g\rightarrow \infty} \frac{(N_{g,1})^{2}}{N_{g,2}}=k$ with $0<k<1$.

\medskip
The numerical data collected is for the moment not sufficient to give us too much confidence on our observations or to state any nice conjecture (especially for the case $d\geq 3$), as happened with Bras-Amorós when she realized that the sequence of the number of numerical semigroups counted by genus had a Fibonacci-like behaviour. Since the algorithms and the implementations have still space to be improved, obtaining more numerical data may be seen as an active goal. 

%All the numerical data collected are insufficient to state new conjectures for the set of generalized numerical semigroups. Anyway it seems possible that such conjectures may exist.

\subsection*{Acknowledgements}
The first and second authors want to thank the amazing hospitality found in the Instituto de Matemáticas de la Universidad de Granada (IEMath-GR). 
Most of the results in this paper where obtained during their stay there.
%Most of this paper has been written during their stay there.

The authors also thank the Centro de Servicios de Informática y Redes de Comunicaciones (CSIRC), Universidad de Granada, for providing the computing time, specially Rafael Arco Arredondo for installing this package and the extra software needed in alhambra.ugr.es, and Santiago Melchor Ferrer for helping in job submission to the cluster.

%\nocite{*}
\bibliographystyle{plain}
\bibliography{biblio}

\begin{thebibliography}{10}

\bibitem{assi2015frobenius}
A.~Assi, P.A. Garc{\'\i}a-S{\'a}nchez, and I.~Ojeda.
\newblock Frobenius vectors, hilbert series and gluings of affine semigroups.
\newblock {\em Journal of Commutative Algebra}, 7(3):317--335, 2015.

\bibitem{ns-app}
Abdallah Assi and Pedro~A. Garc\'{i}a-S\'{a}nchez.
\newblock {\em Numerical semigroups and applications}, volume~1 of {\em RSME
  Springer Series}.
\newblock Springer, 2016.

\bibitem{bras2008fibonacci}
Maria Bras-Amor{\'o}s.
\newblock Fibonacci-like behavior of the number of numerical semigroups of a
  given genus.
\newblock {\em Semigroup Forum}, 76(2):379--384, 2008.

\bibitem{bras2009bounds}
Maria Bras-Amor{\'o}s.
\newblock Bounds on the number of numerical semigroups of a given genus.
\newblock {\em Journal of Pure and Applied Algebra}, 213(6):997--1001, 2009.

\bibitem{bras2012ordinarization}
Maria Bras-Amor{\'o}s.
\newblock The ordinarization transform of a numerical semigroup and semigroups
  with a large number of intervals.
\newblock {\em Journal of Pure and Applied Algebra}, 216(11):2507--2518, 2012.

\bibitem{Analele}
Carmelo Cisto, Gioia Failla, and Rosanna Utano.
\newblock On the generators of a generalized numerical semigroup.
\newblock {\em Analele Univ. “Ovidius”}, 27(1):49--59, 2019.

\bibitem{cox2007ideals}
David Cox, John Little, and Donal O'shea.
\newblock {\em Ideals, varieties, and algorithms}, volume~3.
\newblock Springer, 2007.

\bibitem{numericalsgps}
M.~Delgado, P.~A. Garcia-Sanchez, and J.~Morais.
\newblock {NumericalSgps}, a package for numerical semigroups, {V}ersion
  1.1.11.
\newblock \url{https://gap-packages.github.io/numericalsgps}, Mar 2019.
\newblock Refereed GAP package.

\bibitem{failla2016algorithms}
Gioia Failla, Chris Peterson, and Rosanna Utano.
\newblock Algorithms and basic asymptotics for generalized numerical semigroups
  in $\mathbb{N}^d$.
\newblock {\em Semigroup Forum}, 92(2):460--473, 2016.

\bibitem{GAP}
{GAP} {\textendash} {G}roups, {A}lgorithms, and {P}rogramming, {V}ersion
  4.10.0.
\newblock \url{https://www.gap-system.org}, Nov 2018.

\bibitem{2019arXiv190311028G}
J.~I. {Garc{\'\i}a-Garc{\'\i}a}, I.~{Ojeda}, J.~C. {Rosales}, and
  A.~{Vigneron-Tenorio}.
\newblock On pseudo-frobenius elements of submonoids of \(\mathbb{N}^d\).
\newblock {\em Collect. Math.}, 2019.
\newblock https://doi.org/10.1007/s13348-019-00267-0.

\bibitem{garcia2018extension}
J.I. Garc{\'\i}a-Garc{\'\i}a, D.~Mar{\'\i}n-Arag{\'o}n, and
  A.~Vigneron-Tenorio.
\newblock An extension of wilf’s conjecture to affine semigroups.
\newblock {\em Semigroup Forum}, 96(2):396--408, 2018.

\bibitem{garcia2013affine}
J.I. Garc{\'\i}a-Garc{\'\i}a, M.A. Moreno-Fr{\'\i}as, A.~S{\'a}nchez~R.
  Navarro, and A.~Vigneron-Tenorio.
\newblock Affine convex body semigroups.
\newblock {\em Semigroup Forum}, 87(2):331--350, 2013.

\bibitem{rosales2009numerical}
Jos{\'e}~Carlos Rosales and Pedro~A. Garc{\'\i}a-S{\'a}nchez.
\newblock {\em Numerical semigroups}, volume~20.
\newblock Springer Science \& Business Media, 2009.

\bibitem{rosales2003oversemigroups}
Jos{\'e}~Carlos Rosales, Pedro~A. Garc{\'\i}a-S{\'a}nchez, Juan~Ignacio
  Garc{\'\i}a-Garc{\'\i}a, and J.A.~Jim{\'e}nez Madrid.
\newblock The oversemigroups of a numerical semigroup.
\newblock {\em Semigroup Forum}, 67(1):145--158, 2003.

\bibitem{zhai2013fibonacci}
Alex Zhai.
\newblock Fibonacci-like growth of numerical semigroups of a given genus.
\newblock {\em Semigroup Forum}, 86(3):634--662, 2013.

\end{thebibliography}

%\begin{thebibliography}{20}

%\end{thebibliography}

\end{document}